\documentclass[12pt]{amsart}
\usepackage{amssymb,amsmath,mathrsfs}
\usepackage[margin=2.4cm]{geometry}
\usepackage[colorlinks=true,linkcolor=blue,citecolor=blue]{hyperref}
\usepackage{aliascnt}
\usepackage{cleveref}
\usepackage{tikz}

\def\scal#1#2{\langle #1, #2\rangle}
\def\R#1{\mathbb{R}^{#1}}
\newcommand\alg{\mathbb{A}}

\newcommand\Idm{\mathrm{Idm}}
\newcommand\hIdm{\widehat{\mathrm{Idm}}}
\newcommand\hcub{\mathscr{H}_3}
\newcommand\Fs[2]{[#1,#2]}
\newcommand\so{\mathfrak{so}}
\newcommand\Cartan[1]{\mathcal{C}_{#1}}
\newcommand\Pol{\mathsf{Pol}}
\newcommand\Kill{\mathcal{K}}
\newcommand\Tens{\mathcal{T}}
\DeclareFontFamily{U}{wncy}{}
\DeclareFontShape{U}{wncy}{m}{n}{<->wncyr10}{}
\newcommand\Che{\mathord{\text{\usefont{U}{wncy}{m}{n}Ch}}}
\DeclareMathOperator{\trace}{tr}
\DeclareMathOperator{\re}{Re}
\DeclareMathOperator{\conv}{conv}
\DeclareMathOperator{\supp}{supp}

\DeclareMathOperator{\Sym}{Sym}
\DeclareMathOperator{\Id}{Id}
\DeclareMathOperator{\sgn}{sgn}

\newtheorem{theorem}{Theorem}[section]
\newcommand{\sharedtheorem}[2]{%
  \newaliascnt{#1}{theorem}%
  \newtheorem{#1}[#1]{#2}%
  \aliascntresetthe{#1}%
  \crefname{#1}{#2}{#2s}\Crefname{#1}{#2}{#2s}}
\crefname{theorem}{Theorem}{Theorems}\Crefname{theorem}{Theorem}{Theorems}
\sharedtheorem{proposition}{Proposition}
\sharedtheorem{lemma}{Lemma}
\sharedtheorem{corollary}{Corollary}
\sharedtheorem{conjecture}{Conjecture}
\newtheorem*{theoremA}{Theorem A}
\newtheorem*{theoremB}{Theorem B}
\newtheorem*{theoremC}{Theorem C}
\newtheorem*{theoremD}{Theorem D}
\newtheorem*{theoremE}{Theorem E}
\theoremstyle{definition}
\sharedtheorem{definition}{Definition}
\sharedtheorem{problem}{Problem}
\theoremstyle{remark}
\sharedtheorem{remark}{Remark}
\sharedtheorem{example}{Example}
\numberwithin{equation}{section}
\crefname{corollary}{Corollary}{Corollaries}\Crefname{corollary}{Corollary}{Corollaries}
\crefname{table}{Table}{Tables}\Crefname{table}{Table}{Tables}

\title[Hsiang cubics as spherical Chebyshev polynomials]{Hsiang cubics as spherical Chebyshev polynomials:\\ equioscillation and idempotent quadrature}
\author{Vladimir G. Tkachev}
\address{Department of Mathematics, Link\"oping University, Sweden}
\email{vladimir.tkatjev@liu.se}
\thanks{V.G.T. has been supported by SveFUM and the Stiftelsen L\"angmanska kulturfonden, Grant BA26-2924.}
\date{Draft, October 2026}
\dedicatory{Dedicated to Bj\"orn Gustafsson on the occasion of his 80th birthday.}
\subjclass[2020]{Primary 41A50, 17A30; Secondary 35J94, 33C55, 44A60, 65D32, 53C42}
\keywords{Hsiang algebra, cubic minimal cone, Chebyshev polynomial, equioscillation, Kolmogorov criterion, infinity-Laplacian, spherical harmonics, spherical design, quadrature identity, idempotent}

\begin{document}

\begin{abstract}
A Hsiang cubic is a cubic form on $\R{n}$ whose zero set is a minimal cone; equivalently, it is the cubic form of a commutative nonassociative metrized algebra satisfying a single identity of degree four. We show that the intrinsic geometry of such an algebra, encoded in its idempotents and their Peirce decompositions, is the geometry of a Chebyshev-type extremal problem on the sphere. A Hsiang cubic equioscillates on $S^{n-1}$: its maximum set is the set of normalized idempotents, and on suitable great circles through an idempotent it restricts to the Chebyshev polynomial $T_3$. For the spherical analogue of the problem of least deviation, the minimization of $\|p\|_\infty^2/\|p\|_2^2$ over harmonic cubics, a sharp lower bound based on Kellogg's inequality shows that Cartan's isoparametric cubics are the global minimizers in dimensions $5$, $8$, $14$ and $26$. In general, we compute the ratio of every Hsiang cubic in terms of the Killing form and characterize stationarity by a Kolmogorov criterion: it amounts to a positive measure on the idempotent variety. Such measures are given explicitly in several cases which together cover every Hsiang cubic known to exist; in particular, all of them are stationary. These measures give quadrature identities on the idempotent variety whose moments are invariants of the algebra; they are spherical designs exactly in the Einstein (Killing metrized) case, i.e.\ when the Killing form is proportional to the inner product. Distinguished low-dimensional Hsiang cubics are strict local minimizers, whereas a polar Hsiang cubic is a stationary non-minimizer.
\end{abstract}

\maketitle

\setcounter{tocdepth}{1}
\tableofcontents

\section{Introduction}\label{sec:intro}

\subsection{Motivation}
This paper is about a class of cubic forms for which a hard extremal problem on the sphere, namely how flat a harmonic cubic can be, admits explicit answers: the cubic forms of Hsiang algebras. Their algebraic structure plays a role similar to that of the disk and its images under polynomial maps in the theory of quadrature domains.

The classical Chebyshev polynomial $T_k(t)=\cos(k\arccos t)$ is characterized by \emph{equioscillation}: on $[-1,1]$ all its local extrema have the same absolute value and alternate in sign. Equioscillation is the reason why $2^{1-k}T_k$ deviates least from zero among monic polynomials of degree $k$, and, through the Kolmogorov criterion, it is the prototype of the characterization of best uniform approximations \cite{Kolmogorov-chebyshev,Rivlin,Singer-bestapproximation}. The alternation points of $T_k$ also carry positive quadrature formulas.

The purpose of this paper is to show that the same circle of ideas appears, unexpectedly but very naturally, in the theory of \emph{Hsiang algebras}, a class of commutative algebras attached to cubic minimal cones \cite{Hsiang,Nadirashvili-Tkachev-Vladuts,tkachev2010minimal,Tk10b,Tkachev-jordan,Fox-Tkachev,Tkachev-book}. Here the word \emph{algebra} means a real vector space with a bilinear multiplication which is \emph{not} assumed to be associative; Hsiang algebras are commutative but nonassociative (a non-trivial Hsiang algebra has no unit and does not satisfy the associative law). A \emph{Hsiang cubic} (or Hsiang eigencubic) is a cubic form $u$ on $\R{n}$ satisfying
\begin{equation}\label{hsiangeq}
|Du|^2\Delta u-\tfrac12\scal{Du}{D|Du|^2}=\theta\,|x|^2u,\qquad\theta\in\R{}.
\end{equation}
W.-Y.~Hsiang \cite{Hsiang} observed that the left hand side is $|Du|^3$ times the mean curvature of the level sets of $u$, so that \eqref{hsiangeq} forces the cone $u^{-1}(0)$ to be minimal. The algebraic theory of \eqref{hsiangeq} \cite{tkachev2010minimal,Tk10b,Tkachev-jordan}, \cite[Chap.~6]{Nadirashvili-Tkachev-Vladuts}, \cite{Tkachev-book} starts from the observation that $u(x)=\tfrac16\scal{x^2}{x}$ for a unique commutative algebra structure on $\R{n}$ with an invariant inner product, and that \eqref{hsiangeq} becomes a polynomial identity of degree four for this algebra.

The thesis of this paper is that \emph{the intrinsic idempotent and Peirce geometry of a Hsiang algebra is the geometry of a Chebyshev-type extremal problem on the sphere}. The correspondence is summarized in \Cref{tab:dictionary}.

\begin{table}[ht]
\renewcommand\arraystretch{1.25}\small
\begin{tabular}{l|l|l}
Chebyshev theory & Hsiang algebra $\alg$, cubic form $u$ & \\\hline
$T_k=\cos k\varphi$ & $u=M\cos3\varphi$ on great circles through $\hat c$ & \Cref{the:equi}\\
equal extremal values $\pm1$ & all idempotents have the same length & \Cref{sec:hsiang}\\
alternation set & normalized idempotents $\pm\hIdm$ & \Cref{the:equi}\\
second variation at an extremum & Peirce decomposition of $L(c)$ & \Cref{the:equi}\\
Bernstein--Szeg\H{o} equality & Kellogg equality on $S\cap(\R{}c\oplus\alg_c(-1))$ & \Cref{pro:firstint}\\
$T_k$: equality everywhere & Cartan's cubics (eikonal), unique minimizers & \Cref{the:cartanmin}\\
level of least deviation & $M=\max_{S^{n-1}}u$, \ $-\theta=54M^2$ & \Cref{pro:infeig}\\
Kolmogorov criterion & Kolmogorov measures on $\Idm(\alg)$ & \Cref{the:kolm}\\
quadrature on alternation points & representing measures on $\Idm(\alg)$ & \Cref{the:moments}\\
moments of the quadrature & mass $\frac{|c|}6\trace \Kill$, moments $0$, $\frac{|c|}6\Kill$, $\frac16\Tens$ & \Cref{the:moments}
\end{tabular}
\medskip
\caption{The dictionary between Chebyshev theory and Hsiang algebras.}
\label{tab:dictionary}
\end{table}

We do not develop the whole story here; rather, we prove the basic results along this correspondence and formulate the problems which we think are worth the attention of specialists in approximation theory, potential theory and nonlinear PDE on the one hand, and in nonassociative algebra on the other.

\subsection{The Chebyshev ratio}
Let $\hcub=\hcub(\R{n})$ be the space of harmonic cubic forms on $\R{n}$ and $\sigma$ the normalized surface measure on $S^{n-1}$. For $0\ne p\in\hcub$ we call
\begin{equation}\label{Rdef}
R(p)=\frac{\|p\|_\infty^2}{\|p\|_2^2},\qquad\text{where}\quad \|p\|_\infty=\max_{x\in S^{n-1}}|p(x)|,\quad \|p\|_2^2=\int_{S^{n-1}}p^2\,d\sigma,
\end{equation}
the \emph{Chebyshev ratio} of $p$, and we consider
\begin{equation}\label{kappa0}
\kappa_n:=\min\bigl\{R(p):\ 0\ne p\in\hcub(\R{n})\bigr\}.
\end{equation}
The ratio is invariant under scaling $p\mapsto tp$ and under rotations $p\mapsto p\circ g$, $g\in O(n)$, and the minimum is attained by compactness. The constraint $\|p\|_2=1$ replaces the monic normalization of the classical problem, and $\kappa_n$ measures how flat a harmonic cubic can be on the sphere. We work with the squared ratio because all its values below are rational.

\emph{Why this ratio.} The numerator is canonical: $\|p\|_\infty$ is the spectral norm of the symmetric $3$-tensor $S$ with $p(x)=S(x,x,x)$, by a classical theorem of Banach (see \cite{Agrachev-Kozhasov-Uschmajew}), and $3\|p\|_\infty$ is the Lipschitz constant of $p$ on the unit ball $B$, by Kellogg's theorem (\Cref{rem:kellogg}). The denominator is a choice. Every cubic form splits $O(n)$-invariantly as
\begin{equation}\label{harmdec}
\mathscr P_3=\hcub\oplus|x|^2\mathscr H_1,\qquad p=h+|x|^2\ell,\quad \ell=\frac{\Delta p}{2(n+2)},
\end{equation}
into its \emph{harmonic part} $h$ and its \emph{trace part} $|x|^2\ell$ \cite{Axler-Bourdon-Ramey}, and an invariant norm may weigh the two parts independently. The two natural choices, the $L^2(S^{n-1})$-norm and the Bombieri norm $\|p\|_B$ (the Frobenius norm of $S$), differ on the trace part (\Cref{rem:bombieri}) but are proportional on $\hcub$: by \eqref{bombieri}, $R(p)=\frac{N_n}6\,(\|p\|_\infty/\|p\|_B)^2$, where
\begin{equation}\label{Nn}
N_n:=n(n+2)(n+4)
\end{equation}
is the normalizing constant of the Fischer inner product on $\hcub$ (\Cref{lem:fischer}). Since Hsiang cubics are harmonic (\Cref{the:exact}), both choices lead to the same problem. We use the $L^2$-normalization on the analytic side (the reproducing kernel of $\hcub$ and the Kolmogorov measures of \Cref{sec:kolm}) and the Bombieri normalization on the algebraic side (tensors, Waring decompositions and the representing measures of \Cref{sec:waring}); for instance, $\|u\|_B^2=\frac1{36}\trace \Kill$ for a Hsiang cubic with Killing operator $\Kill$, see \eqref{killing}.

\emph{Context.} Comparisons of the uniform and quadratic norms of spherical harmonics are classical. The maximum of $R$ over spherical harmonics of a given degree is the square of the sharp constant in the Nikolskii inequality and is attained by the zonal harmonics (\Cref{sec:problem}, \cite[Chap.~1]{Dai-Xu}), and the size of $R$ for typical (random) spherical harmonics is well understood \cite{Burq-Lebeau}. The opposite problem of minimizing $R$, i.e.\ of finding the flattest harmonics, is the spherical counterpart of the search for flat trigonometric polynomials \cite{Kahane-unimodular}. Without the harmonicity constraint, and with the Bombieri norm, it becomes the problem of the best rank-one approximation ratio of symmetric tensors, which was solved by Agrachev, Kozhasov and Uschmajew \cite{Agrachev-Kozhasov-Uschmajew} for binary forms and for ternary cubics; remarkably, their extremal forms are again Chebyshev polynomials (\Cref{rem:aku,pro:aku}). For general (non-symmetric) trilinear forms the corresponding ratio was studied by Cobos, K\"uhn and Peetre \cite{Cobos-Kuhn-Peetre-Duke,Kuhn-Peetre-embedding}, again with distinguished extremal forms (\Cref{pro:peetre}). Problem \eqref{kappa0} is the harmonic counterpart of their problem; we are not aware of earlier results on it. The question is what the minimizers look like.

\subsection{Main results}
We begin with a result which solves the problem completely in some dimensions. It concerns arbitrary harmonic cubics and is independent of the theory of Hsiang algebras: its proof (\Cref{sec:cartanmin}) uses only the Dirichlet energy of $p$ in the unit ball, Kellogg's inequality \eqref{kellogg} and the classification of eikonal cubics.

\begin{theoremA}[Cartan's cubics are global minimizers; \Cref{sec:cartanmin}]
For every $n\ge2$ and every non-zero harmonic cubic $p$ on $\R n$, $R(p)\ge\frac{n+4}3$, with equality if and only if $p$ satisfies the eikonal equation $|Dp|^2=9\|p\|_\infty^2|x|^4$. Consequently $\kappa_n=\frac{n+4}3$ exactly for $n\in\{2,5,8,14,26\}$, the minimizers being $\re z^3$ and Cartan's isoparametric cubics $\Cartan d$ (\Cref{sec:hsiang}); in particular $\kappa_5=3$.
\end{theoremA}

Theorem~A settles the problem only in those dimensions in which a harmonic cubic can attain Kellogg's bound at every point of the sphere. In all other dimensions the extremal cubics are flat in a weaker sense, and a finer structure is needed to describe them. Hsiang cubics are the natural candidates: they attain Kellogg's bound on whole spheres through each idempotent (\Cref{pro:firstint}). The remaining results describe this structure from four points of view: equioscillation and the exact value of the ratio, the first variation, where stationarity becomes a quadrature problem on the idempotent variety, the existence of such quadratures, and the local versus global behaviour of $R$.

Throughout, a \emph{Hsiang cubic} is the cubic form $u(x)=\frac16\scal{x^2}{x}$ of a non-trivial Hsiang algebra $\alg$ on $\R{n}$, normalized as in \Cref{sec:hsiang}, so that
\begin{equation}\label{Mnorm}
M:=\max_{S^{n-1}}u=\frac1{3\sqrt3}=\frac1{6|c|},\qquad |c|^2=\frac34\quad\text{for every } c\in\Idm(\alg).
\end{equation}
The constant $M$ is the level of least deviation in all results below. We write $\hat c=c/|c|$ for $c\in\Idm(\alg)$, $L(x)y=xy$ for the multiplication operators, and
\begin{equation}\label{killing}
\Kill=\sum_{i=1}^nL(e_i)^2,\qquad \scal{Ax}{y}=\trace L(x)L(y),
\end{equation}
for the \emph{Killing operator} of $\alg$, where $\{e_i\}$ is any orthonormal basis. Its trace $\trace \Kill=\sum_{i,j}|e_ie_j|^2$ is the sum of the squares of all structure constants of $\alg$; it measures ``how much multiplication'' the algebra has. The next result makes the word ``Chebyshev'' literal: the alternation set of $u$ is its idempotent variety, and the classical Chebyshev polynomial $T_3$ appears on great circles through it.

\begin{theoremB}[Equioscillation and the Chebyshev ratio; \Cref{sec:equi,sec:problem}]
Let $u$ be a Hsiang cubic on $\R{n}$ with algebra $\alg$. Then:
\begin{enumerate}
\item[(i)] $u$ attains the value $M$ exactly on $\hIdm=\{\hat c:\ c\in\Idm(\alg)\}$, and $u|_{S^{n-1}}$ has no other local maxima;
\item[(ii)] $u=M\cos3\varphi=M\,T_3(\cos\varphi)$ on every great circle through $\hat c$ in a direction of the Peirce space $\alg_c(-1)$;
\item[(iii)] the Chebyshev ratio of $u$ is $R(u)=\dfrac{2N_n}{9\trace \Kill}$, where $\Kill$ is the Killing operator \eqref{killing}.
\end{enumerate}
\end{theoremB}

For a general cubic form neither $\|p\|_\infty$ nor its maximum set is accessible: computing the spectral norm of a symmetric $3$-tensor is NP-hard \cite{Hillar-Lim}. For a Hsiang cubic both are given by the algebra: the level is $M$, the alternation set is $\Idm(\alg)$, and $R(u)$ is a rational number determined by the Killing trace. In this sense the Hsiang cubics are model examples for the Chebyshev problem.

As in the classical theory, the alternation set is what makes the first variation of the sup-norm computable. For Hsiang cubics this turns stationarity into a positivity condition on the idempotent variety, i.e.\ into a moment problem on $\Idm(\alg)$. Recall that the harmonic part of the cube $\scal xa^3$ is the zonal harmonic $Z_a$, see \eqref{zonal}. A positive measure $\mu$ on $\Idm(\alg)$ is called a \emph{Kolmogorov measure} of $u$ if $u$ is the $\mu$-average of the zonal harmonics $Z_{\hat c}$, and a \emph{representing measure} of $u$ if $u$ is the $\mu$-average of the cubes themselves, $u(x)=\int\scal x{\hat c}^3d\mu(c)$; in the language of the truncated moment problem, the cubic moment tensor of $\mu$ is then $\frac16\Tens$, where $\Tens(x,y,z)=\scal{xy}z$. Since $u$ is harmonic, the trace parts of the cubes average to zero, so every representing measure is a Kolmogorov measure (\Cref{sec:waring}).

\begin{theoremC}[Kolmogorov criterion and idempotent quadrature; \Cref{sec:kolm,sec:waring}]
Let $u$ be a Hsiang cubic on $\R{n}$ with algebra $\alg$. Then:
\begin{enumerate}
\item[(i)] $u$ is a stationary point of the Chebyshev ratio $R$ on $\hcub$ if and only if $u$ admits a Kolmogorov measure;
\item[(ii)] every representing measure $\mu$ of $u$ has mass $\frac{|c|}6\trace \Kill$, first moment $0$, second moment $\frac{|c|}6\Kill$ and third moment $\frac16\Tens$, where $\Kill$ is the Killing operator \eqref{killing};
\item[(iii)] in particular, $\mu/\mu(\Idm(\alg))$ is a spherical $2$-design if and only if $\alg$ is an Einstein Hsiang algebra.
\end{enumerate}
\end{theoremC}

A representing measure is thus a quadrature identity for $u$ with nodes on $\Idm(\alg)$, and the analogy with quadrature domains extends to the shape of the measures. A zonal harmonic $Z_a$ has a one-point quadrature, the analogue of the mean value property of the disk (\Cref{lem:zonal}). Hsiang cubics with a Steiner decomposition have quadratures with finitely many nodes (\Cref{sec:psts}). The uniform measures of \Cref{sec:existence} (Theorem~D below) are spread over the whole idempotent variety, which then plays the role of a \emph{mother body}, or skeleton, of $u$, as the focal segment does for an ellipse \cite{Gustafsson-motherbodies,Gustafsson-Shapiro}.

Theorem~C thus reduces stationarity to the existence of a positive measure on $\Idm(\alg)$ with prescribed moments. By the classification, a Hsiang algebra is either polar or exceptional, and the latter occur only in finitely many dimensions. For the infinite family of polar algebras the moment problem has an explicit uniform solution. For the exceptional algebras it is solved either by symmetry or by a finite combinatorial quadrature, and together these cover all exceptional Hsiang algebras known to exist.

\begin{theoremD}[Existence of representing measures; \Cref{sec:existence}]
Let $u$ be a Hsiang cubic on $\R{n}$ with algebra $\alg$.
\begin{enumerate}
\item[(i)] If $\alg\in\Pol(m_0,m_1)$ is polar, with polar decomposition $\mathbb V_0\oplus\mathbb V_1$, then the uniform measure on the sphere bundle $\Idm(\alg)$ is, up to a positive factor, a representing measure of $u$. In particular, $u$ is stationary, $\Kill=m_1\Id_{\mathbb V_0}\oplus2m_0\Id_{\mathbb V_1}$, and
$$
R(u)=\frac{2N_n}{27\,m_0m_1}.
$$
\item[(ii)] If the group $G$ of orthogonal automorphisms of $\alg$ acts transitively on $\Idm(\alg)$ and every $G$-invariant cubic form is a multiple of $u$, then the $G$-invariant probability measure on $\Idm(\alg)$ is, up to a positive factor, a representing measure of $u$. This applies to Cartan's cubics $\Cartan d$, to $\det X$ on $M_3(\R{})$ and to $\re\det z$ on the space $\mathrm{Sym}_3(\mathbb C)\cong\R{12}$ of complex symmetric matrices.
\item[(iii)] If $u$ is a sum $\sum_b\pm x_{b_1}x_{b_2}x_{b_3}$ of monomials over the blocks of a partial Steiner triple system (in a suitable orthonormal basis), then $u$ has a representing measure with equal atoms at $4|B|$ explicit idempotents. This applies to the mutants, to $\det X$, to the $21$-dimensional exceptional algebra $\mathrm{trip}(\operatorname{Im}\mathbb O)$ and, more generally, to every exceptional Hsiang algebra with $n_{-1/2}\ne0$ known to exist, with the single exception of the $12$-dimensional algebra in \textup{(ii)}, which has no Steiner decomposition, as announced in \cite{HsiangStrongRigidity}.
\end{enumerate}
Consequently, every Hsiang cubic known to exist is stationary. Stationarity is open only for hypothetical exceptional Hsiang algebras whose Peirce profiles are admissible but neither realized nor excluded; these lie in finitely many dimensions.
\end{theoremD}

Stationarity is only a first-order condition, which the Hsiang cubics share with the maximizers of $R$, the zonal harmonics. Whether a stationary Hsiang cubic is a genuine local minimizer depends on the harmonic cubics $h$ that vanish on $\hIdm$, where the first variation is zero. Infinitesimal rotations of $u$ are always among them; if they are the only ones, $u$ is a strict local minimizer, and otherwise $R$ may decrease at second order. The answer depends on the algebra.

\begin{theoremE}[Local minimizers and saddles; \Cref{sec:local}]
Local minimality of a Hsiang cubic depends on the algebra:
\begin{enumerate}
\item[(i)] The cubic $x_1x_2x_3$ on $\R3$ ($\Pol(1,2)$) and the Lawson cubic on $\R4$ ($\Pol(2,2)$) are strict local minimizers of $R$ on $\hcub$ modulo rotations and scaling. (On $\R5$, Cartan's cubic $\Cartan1$ is even the global minimizer by Theorem~A.)
\item[(ii)] The cubic $u=\frac12v_0(y_1^2+y_2^2-y_3^2-y_4^2)$ of $\Pol(1,4)$ on $\R5$ is stationary but not a local minimizer: $R(\cos\theta\,u+\sin\theta\,u_L)<R(u)$ for all $\theta\notin\frac\pi2\mathbb Z$, where $u_L$ is the Lawson cubic ($\Pol(2,2)$) on the hyperplane $e_0^\perp\subset\R5$.
\end{enumerate}
\end{theoremE}

\subsection{What is proved and what is conjectured}
Theorems A--E are proved in full. Theorem~A solves the global problem \eqref{kappa0} for $n=5,8,14,26$, and Theorems B--E give local information in the other dimensions: Theorem~E provides local evidence for \Cref{conj:kappa} below, according to which $\kappa_3=\frac{35}9$ and $\kappa_4=\frac{32}9$ are attained, up to rotations and scaling, by $x_1x_2x_3$ ($\Pol(1,2)$) and the Lawson cubic ($\Pol(2,2)$). For $n=3$ the result of Agrachev, Kozhasov and Uschmajew on general ternary cubics gives $\kappa_3\ge\frac52$ (\Cref{pro:aku}), better than the bound $\frac73$ of Theorem~A. The three central open questions suggested by our results are formulated in \Cref{sec:open}: whether every Hsiang algebra, including the hypothetical exceptional ones, admits a representing measure on its idempotent variety (\Cref{prob:waring}), whether the global minimizers of \eqref{kappa0} are Hsiang cubics (\Cref{prob:global}), and whether Hsiang cubics are characterized by equioscillation or stationarity (\Cref{prob:converse}).

\emph{Why Hsiang cubics.} Since the global minimizers of $R$ are known only in the dimensions of Theorem~A, one may ask why Hsiang cubics, rather than other cubic forms with a rich symmetric or combinatorial structure, should be singled out. The answer is the Kolmogorov criterion. The first variation of the uniform norm lives on the maximum set of $|u|$ on the sphere, so a quadrature identity for $u$ says something about $R$ only if its nodes are maximum points. Quadrature identities as such are cheap: every cubic form is a positive sum of cubes $\scal xa^3$, since $-\scal xa^3=\scal x{-a}^3$, and every cubic form with a Steiner decomposition (\Cref{sec:psts}) has one with nodes at the vertices of explicit regular tetrahedra. What distinguishes the Hsiang cubics is that their quadrature nodes are forced onto the alternation set: all idempotents have the same length, so every idempotent is a maximum point (Theorem~B). The classical situation is the same: every polynomial admits positive quadrature formulas, but only for $T_k$ are the nodes its alternation points. For cubic forms with a Steiner decomposition this becomes quantitative: $R(u)\ge N_n/(27|B|)$, with equality exactly when the tetrahedral nodes are maximum points, as they are for Hsiang cubics (\Cref{rem:nonhsiang}). Nor is the combinatorics alone sufficient: the sign decoration of the blocks, which is of a purely algebraic nature, matters as well \cite{HsiangStrongRigidity}.

\subsection{A second entrance: the \texorpdfstring{$\infty$}{infinity}-Laplacian}
The second term in \eqref{hsiangeq} is the $\infty$-Laplacian $\Delta_\infty u=\scal{D^2u\,Du}{Du}=\frac12\scal{Du}{D|Du|^2}$, the Euler--Lagrange operator of the $L^\infty$-variational problem of minimizing $\|Du\|_{L^\infty}$ discovered by G.~Aronsson \cite{Aronsson-extension,Aronsson-pde}, see also \cite{Jensen-lipschitz,Crandall-Evans-Gariepy,Aronsson-Crandall-Juutinen,Peres-Schramm-Sheffield-Wilson}. Since every non-trivial Hsiang cubic is harmonic (\Cref{the:exact}), \eqref{hsiangeq} reduces to the nonlinear eigenvalue equation $\Delta_\infty u=-\theta|x|^2u$: the non-trivial Hsiang cubics are exactly the \emph{harmonic $\infty$-eigencubics} of \Cref{def:infeig}. The eigenvalue is the squared Chebyshev level, $-\theta=54M^2$ with $M=\max_{S^{n-1}}u$ (\Cref{pro:infeig}), and by Kellogg's theorem the Chebyshev problem \eqref{kappa0} is itself a minimal Lipschitz problem in Aronsson's sense (\Cref{rem:kellogg}). This suggests deeper relations between Hsiang theory and $L^\infty$-variational problems, which we only touch upon in \Cref{sec:infty}.

\subsection{Organization of the paper}
\Cref{sec:hsiang} collects the few facts on Hsiang algebras that we need, with references to \cite[Chap.~6]{Nadirashvili-Tkachev-Vladuts} and \cite{Tkachev-book}: exactness, the equal length of idempotents, the Peirce decomposition, and the multiplication of $(-1)$-Peirce vectors. \Cref{sec:infty} treats the $\infty$-Laplacian and Kellogg's inequality, and \Cref{sec:cartanmin} proves Theorem~A. In \Cref{sec:equi} we prove the equioscillation property and in \Cref{sec:problem} we compute the Chebyshev ratio. \Cref{sec:kolm} contains the Kolmogorov criterion, \Cref{sec:waring} the moment identities and \Cref{sec:existence} the existence of representing measures. Local minimality is studied in \Cref{sec:local}, and \Cref{sec:open} contains concluding remarks and open problems.

\subsection{Acknowledgements}
The author thanks Bj\"orn Gustafsson for many discussions over the years and for the long-standing collaboration on quadrature domains and the meromorphic resultant \cite{Gustafsson-Tkachev-resultant,Gustafsson-Tkachev-multisheeted}. During the last years we discussed on several occasions potential-theoretic contexts in which Hsiang algebras and moments of higher degree polynomials could be relevant; these discussions were important in shaping the present work. The author is also grateful to Bj\"orn Gustafsson for drawing his attention to the work of J.~Peetre on cubic forms. Finally, the author is grateful to Khazhgali Kozhasov for discussions during his visit to Link\"oping in May 2026, on possible interactions between the theory of Hsiang algebras and tensor approximation, and on the Bombieri norm and related constructions.

\subsection{Notation}
\begin{tabbing}
\hspace{32mm}\=\hspace*{120mm}\kill
$\scal xy$, $|x|$\> the standard inner product and norm of $\R n$\\
$S=S^{n-1}$, $B$\> the unit sphere and the unit ball of $\R n$\\
$\sigma$\> the rotation invariant probability measure on $S$\\
$\|p\|_\infty$, $\|p\|_2$\> the uniform norm and the $L^2(S,\sigma)$-norm, see \eqref{Rdef}\\
$Du$, $D^2u$, $\Delta u$\> the gradient, the Hessian and the Laplacian of $u$\\
$\Delta_\infty u$\> the $\infty$-Laplacian $\scal{D^2u\,Du}{Du}$\\
$\mathscr P_3$, $\mathscr H_1$\> the cubic forms and the linear forms on $\R n$\\
$\hcub=\hcub(\R n)$\> the harmonic cubic forms on $\R n$\\
$R(p)$, $\kappa_n$\> the Chebyshev ratio and its minimum, see \eqref{Rdef}, \eqref{kappa0}\\
$N_n$\> the constant $n(n+2)(n+4)$, see \eqref{Nn}\\
$\Fs pq$, $\|p\|_B$\> the Fischer inner product and the Bombieri norm (\Cref{sec:problem})\\
$Z_a$, $K_a$\> the zonal harmonic and the reproducing kernel of $\hcub$, see \eqref{zonal}\\
$\Che_{3,m}$\> the cubic Chebyshev form in $m$ variables (\Cref{rem:aku})\\[1mm]
$\alg$, $\alg(u)$\> an algebra; the metrized algebra of a cubic form $u$\\
$L(x)$\> the multiplication operator $L(x)y=xy$\\
$\Tens$\> the multiplication tensor $\Tens(x,y,z)=\scal{xy}z$, see \eqref{prod}\\
$\Kill$\> the Killing operator, see \eqref{killing}\\
$\alg_c(t)$, $\mathsf{pr}(\alg)$\> the Peirce subspaces and the Peirce profile, see \eqref{peirce}, \eqref{profile}\\
$\theta$, $\lambda$, $M$\> the Hsiang weight, the $\infty$-eigenvalue and the Chebyshev level \eqref{Mnorm}\\
$\Pol(m_0,m_1)$\> the polar algebras of type $(m_0,m_1)$ (\Cref{def:polar})\\
$\Cartan d$\> Cartan's cubic on $\R{3d+2}$, $d\in\{1,2,4,8\}$
\end{tabbing}

\section{Preliminaries on Hsiang algebras}\label{sec:hsiang}

\subsection{Metrized algebras}
By an algebra we mean a finite-dimensional real vector space $\alg$ with a bilinear multiplication $(x,y)\mapsto xy$; neither commutativity nor associativity is assumed in general. For $x\in\alg$ the multiplication operator is $L(x)y=xy$.

\begin{definition}\label{def:metrized}
A bilinear form $Q$ on an algebra $\alg$ is called \emph{invariant} if $Q(xy,z)=Q(x,yz)$ for all $x,y,z\in\alg$. A \emph{metrized algebra} is a pair $(\alg,Q)$, where $\alg$ is a commutative algebra and $Q$ is a nondegenerate symmetric invariant bilinear form on $\alg$ \cite{Bordemann}, \cite[Chap.~6]{Nadirashvili-Tkachev-Vladuts}. It is called \emph{Euclidean} if $Q$ is positive definite.
\end{definition}

In what follows $Q=\scal{}{}$ is the standard inner product of $\R{n}$. In a Euclidean metrized algebra every operator $L(x)$ is self-adjoint. An \emph{idempotent} is an element $c\ne0$ with $c^2=c$; we denote by $\Idm(\alg)$ the set of idempotents and put $\hat c=c/|c|$.

\subsection{The metrized algebra of a cubic form}
Let $u$ be a cubic form on $\R{n}$ and let $\Tens$ be the unique symmetric trilinear form with $u(x)=\frac16\Tens(x,x,x)$. There is a unique commutative multiplication on $\alg=\R{n}$ such that
\begin{equation}\label{prod}
\scal{xy}{z}=\Tens(x,y,z)=D^3u(x,y,z),\qquad x,y,z\in\alg .
\end{equation}
By the symmetry of $\Tens$ the inner product is invariant, so that $\alg=\alg(u)$ is a Euclidean metrized algebra, called the \emph{metrized algebra of $u$}. Conversely, every Euclidean metrized algebra on $\R{n}$ is obtained in this way from the cubic form $u(x)=\frac16\scal{x^2}{x}$. The following identities are immediate from \eqref{prod}:
\begin{equation}\label{dict}
u(x)=\tfrac16\scal{x^2}{x},\qquad Du(x)=\tfrac12x^2,\qquad D^2u(x)=L(x),\qquad \Delta u(x)=\trace L(x).
\end{equation}
Following D.~Fox \cite{Fox-simplicial}, the algebra $\alg$ is called \emph{exact} if $\trace L(x)=0$ for all $x\in\alg$; by \eqref{dict}, $\alg(u)$ is exact if and only if $u$ is harmonic. The \emph{Killing form} of $\alg$ is $\tau(x,y)=\trace L(x)L(y)$, and the \emph{Killing operator} $\Kill$ is the self-adjoint operator with $\scal{Ax}{y}=\tau(x,y)$; in an orthonormal basis $\{e_i\}$ it is given by \eqref{killing}, and
\begin{equation}\label{trA}
\trace \Kill=\sum_{i,j}|e_ie_j|^2=\sum_{i,j,k}\Tens(e_i,e_j,e_k)^2=:\|\Tens\|^2.
\end{equation}

\begin{lemma}\label{lem:critical}
Let $u\not\equiv0$ be a cubic form and $\alg=\alg(u)$. A point $x\in S$ is a critical point of $u|_S$ with $u(x)\ne0$ if and only if $c=x/(6u(x))$ is an idempotent. In particular, $\Idm(\alg)\neq\emptyset$, and the maximum of $u$ on $S$ is attained at points $\hat c$, $c\in\Idm(\alg)$, with $\max_Su=\frac1{6|c|}$.
\end{lemma}

\begin{proof}
By Lagrange's rule and \eqref{dict}, $x\in S$ is critical if and only if $\frac12x^2=\mu x$ for some $\mu\in\R{}$; pairing with $x$ gives $\mu=\frac12\scal{x^2}{x}=3u(x)$. Thus $x^2=6u(x)x$, and for $u(x)\ne0$ this is equivalent to $c^2=c$ with $c=x/(6u(x))$. Conversely, $u(\hat c)=\scal{c^2}{c}/(6|c|^3)=1/(6|c|)$. Since $u$ is odd and $u\not\equiv0$, its maximum on $S$ is positive, and the maximum point is a critical point with $u\ne0$.
\end{proof}

\Cref{lem:critical} is the simplest instance of the analytic method used throughout the theory: the order structure of $\R{}$ and the compactness of the sphere produce idempotents.

\subsection{Hsiang algebras}
By \eqref{dict}, $|Du|^2=\frac14\scal{x^2}{x^2}$ and $\Delta_\infty u=\frac14\scal{x^2}{x^3}$, where $x^3=x\,x^2$. Therefore $u$ satisfies \eqref{hsiangeq} if and only if
\begin{equation}\label{hsiangid}
\scal{x^2}{x^2}\trace L(x)-\scal{x^2}{x^3}=\tfrac23\theta\,\scal xx\scal{x^2}{x},\qquad x\in\alg .
\end{equation}

\begin{definition}\label{def:hsiang}
A Euclidean metrized algebra $\alg$ with $\alg\alg\ne\{0\}$ is called a \emph{Hsiang algebra} with \emph{weight} $\theta$ if it satisfies \eqref{hsiangid} \cite[Chap.~6]{Nadirashvili-Tkachev-Vladuts}, \cite{Tkachev-book}. It is called \emph{trivial} if $\dim\alg\alg=1$; this happens if and only if $u=a\scal{q}{x}^3$ for some $q\in\alg$, $a\ne0$.
\end{definition}

\begin{theorem}[Exactness \cite{Tk10b}, {\cite[Chap.~6]{Nadirashvili-Tkachev-Vladuts}}]\label{the:exact}
A non-trivial Hsiang algebra is exact. Equivalently, every Hsiang cubic which is not congruent to $ax_1^3$ is harmonic.
\end{theorem}

For an exact Hsiang algebra, \eqref{hsiangid} reduces to $\scal{x^2}{x^3}=-\frac23\theta|x|^2\scal{x^2}{x}$. Evaluating this identity at an idempotent gives $\theta<0$ and $|c|^2=-\frac{3}{2\theta}$ for \emph{every} $c\in\Idm(\alg)$. Since a similarity of $\alg$ rescales $\theta$, we may assume that $\alg$ is \emph{normalized}, i.e.\ $\theta=-2$ and $|c|^2=\frac34$. \emph{From now on $\alg$ denotes a non-trivial normalized Hsiang algebra, and $u=\frac16\scal{x^2}{x}$ its cubic form, called a Hsiang cubic.}

\begin{theorem}[Peirce decomposition {\cite[Chap.~6]{Nadirashvili-Tkachev-Vladuts}}, \cite{Tkachev-book}]\label{the:peirce}
Let $c\in\Idm(\alg)$. Then $L(c)$ has the simple eigenvalue $1$ with eigenvector $c$, and
\begin{equation}\label{peirce}
\alg=\R{}c\oplus\alg_c(-1)\oplus\alg_c(-\tfrac12)\oplus\alg_c(\tfrac12),\qquad \alg_c(t)=\ker(L(c)-t),
\end{equation}
is an orthogonal decomposition. Moreover, $y_1y_2=-\frac43\scal{y_1}{y_2}c$ for all $y_1,y_2\in\alg_c(-1)$. The dimensions $n_t=\dim\alg_c(t)$ do not depend on $c$; the triple
\begin{equation}\label{profile}
\mathsf{pr}(\alg)=(n_{-1},\,n_{-1/2},\,n_{1/2})
\end{equation}
is called the \emph{Peirce profile} of $\alg$.
\end{theorem}

The Peirce profile determines the dimension: $n=3n_{-1}+2n_{-1/2}-1$ and $n_{1/2}=2n_{-1}+n_{-1/2}-2$ \cite{Nadirashvili-Tkachev-Vladuts,Tkachev-book}. In particular $L(c)$ always has a negative eigenvalue; since $L(c)^2=L(c^2)=L(c)$ for an idempotent $c$ of an associative commutative algebra, a non-trivial Hsiang algebra is never associative. We shall also use the fact that $\Idm(\alg)$ is a compact real analytic submanifold of $\alg$ of dimension $n_{1/2}$, with tangent space $\alg_c(\frac12)$ at $c$ \cite{Tkachev-book}.

\begin{definition}\label{def:einstein}
A Hsiang algebra is called an \emph{Einstein Hsiang algebra} if its Killing form is proportional to the inner product, i.e.\ $\Kill=\nu\Id$ for some $\nu\in\R{}$ (then $\nu=\trace \Kill /n>0$).
\end{definition}

The name refers to the analogy between the Killing form and the Ricci tensor; in the terminology of Fox \cite[p.~4]{Fox-simplicial} such algebras are \emph{Killing Einstein}. We keep the word ``Hsiang'' in order to avoid confusion with other uses of the term ``Einstein algebra''.

\subsection{Polar and exceptional Hsiang algebras}\label{sec:polar}
Hsiang algebras fall into two classes: an infinite family constructed from Clifford systems, and finitely many exceptional ones.

\begin{definition}\label{def:polar}
A \emph{symmetric Clifford system} of size $(m_0,m_1)$ is a family $P_1,\dots,P_{m_0}$ of symmetric $m_1\times m_1$ matrices such that
\begin{equation}\label{cliffsys}
P_iP_j+P_jP_i=2\delta_{ij}\Id,\qquad 1\le i,j\le m_0,
\end{equation}
where for $m_0=1$ we also require $\trace P_1=0$. With $P(y)=\sum_iy_iP_i$, so that $P(y)^2=|y|^2\Id$, the cubic form
\begin{equation}\label{clifford}
u(y,z)=\tfrac12\scal{P(y)z}{z},\qquad (y,z)\in\R{m_0}\times\R{m_1}=\R{n},
\end{equation}
is a Hsiang cubic \cite{tkachev2010minimal}, \cite[Chap.~6]{Nadirashvili-Tkachev-Vladuts}. Its algebra is called a \emph{polar algebra of type $(m_0,m_1)$}, and the orthogonal splitting $\mathbb V_0\oplus\mathbb V_1=\R{m_0}\oplus\R{m_1}$ is its \emph{polar decomposition}. We denote by $\Pol(m_0,m_1)$ the class of polar algebras of type $(m_0,m_1)$. A Hsiang algebra which is not similar to a polar one is called \emph{exceptional}.
\end{definition}

For $y\ne0$ the eigenvalues of $P(y)$ are $\pm|y|$, both of multiplicity $\frac{m_1}2$; in particular $m_1$ is even. The type determines the algebra up to isomorphism, except when $m_0\equiv1\bmod4$, $m_0\ge5$, and $\R{m_1}$ is the sum of at least two irreducible Clifford modules: then inequivalent Clifford systems \cite{FKM} give non-isomorphic polar algebras of the same type, the first instance being $(m_0,m_1)=(5,16)$, $n=21$. All quantities considered in this paper depend only on the type, and every specific polar algebra used below is unique in its type; we therefore speak of \emph{the} algebra $\Pol(m_0,m_1)$ in these cases. For instance, the Peirce profile of every $\alg\in\Pol(m_0,m_1)$ is
$$
\mathsf{pr}(\alg)=\bigl(m_0-1,\ \tfrac{m_1}2-m_0+2,\ \tfrac{m_1}2+m_0-2\bigr).
$$
Exceptional Hsiang algebras exist only in finitely many dimensions \cite[Chap.~6]{Nadirashvili-Tkachev-Vladuts}, \cite{Tkachev-book}. The polar algebras $\Pol(k,2k)$, $k\in\{1,2,4,8\}$, are called \emph{mutants}; they are realized by $u=\re\bigl((x_1x_2)x_3\bigr)$ on $\mathbb K_k^3$, where $\mathbb K_k$ is the normed division algebra of dimension $k$. The Einstein Hsiang algebras are exactly the exceptional Hsiang algebras and the mutants \cite{Tkachev-book}.

\subsection{Examples}
We list the Hsiang cubics used below, with the normalization \eqref{Mnorm}; we refer to \cite[Chap.~6]{Nadirashvili-Tkachev-Vladuts}, \cite{Tkachev-book}, \cite{Fox-Tkachev} for their algebras. Among them, the polar algebras $\Pol(1,2)$, $\Pol(2,2)$ and $\Pol(1,4)$, together with $\Pol(2,8)$ on $\R{10}$ (\Cref{tab:ratio}), serve as model examples throughout the paper.
\begin{enumerate}
\item The cubic $u=\frac1{3\sqrt3}(x_1^3-3x_1x_2^2)=\frac1{3\sqrt3}\re(x_1+ix_2)^3$ on $\R2$. Its idempotents are $c_k=\frac{\sqrt3}2(\cos\frac{2\pi k}3,\sin\frac{2\pi k}3)$, $k=0,1,2$.
\item The cubic $u=x_1x_2x_3$ on $\R3$. The multiplication is $e_ie_j=e_k$ for $\{i,j,k\}=\{1,2,3\}$ and $e_i^2=0$, and $\Idm(\alg)=\{\frac12\varepsilon:\ \varepsilon\in\{\pm1\}^3,\ \varepsilon_1\varepsilon_2\varepsilon_3=1\}$ is the vertex set of a regular tetrahedron. Its algebra is $\Pol(1,2)$, the smallest mutant: $x_1x_2x_3=\frac12x_1\scal{P_1z}{z}$ with $z=(x_2,x_3)$ and $P_1=\bigl(\begin{smallmatrix}0&1\\1&0\end{smallmatrix}\bigr)$.
\item The \emph{Lawson cubic} $u=v_0v_2v_3+\frac12(v_2^2-v_3^2)v_1$ on $\R4$, cf.\ \cite{Lawson-completeminimal}, \cite[Chap.~6]{Nadirashvili-Tkachev-Vladuts}. Its algebra is $\Pol(2,2)$, and its idempotents form the circle
\begin{equation}\label{lawsonidm}
c(\varphi)=\bigl(\tfrac12\sin\varphi,\ \tfrac12\cos\varphi,\ \tfrac1{\sqrt2}\cos\tfrac\varphi2,\ \tfrac1{\sqrt2}\sin\tfrac\varphi2\bigr),\qquad\varphi\in[0,4\pi).
\end{equation}
\item The cubic of $\Pol(1,4)$,
\begin{equation}\label{pol14}
u=\tfrac12v_0(y_1^2+y_2^2-y_3^2-y_4^2)\qquad\text{on }\R5=\R{}\times\R4,
\end{equation}
with $P_1=\mathrm{diag}(1,1,-1,-1)$, the polar competitor of Cartan's cubic $\Cartan1$ in the same dimension. Its idempotents form two circles, one in each of the half-spaces $v_0>0$ and $v_0<0$ (\Cref{ex:polar14}).
\item \emph{Cartan's isoparametric cubics} \cite{Cartan-isoparametric}. Let $\mathscr H_3(\mathbb K_d)$ be the Jordan algebra of Hermitian $3\times3$ matrices over $\mathbb K_d\in\{\R{},\mathbb C,\mathbb H,\mathbb O\}$, $d=\dim\mathbb K_d$, with the inner product $\scal XY=\re\trace(XY)$, and let $\alg$ be its traceless part, $n=3d+2$. Then $u=\frac{\sqrt2}{3}\trace X^3$, and
$$
\Idm(\alg)=\bigl\{\tfrac{3\sqrt2}{4}\bigl(e-\tfrac13I\bigr):\ e\ \text{a primitive idempotent of}\ \mathscr H_3(\mathbb K_d)\bigr\}
$$
is a Veronese variety; for $d=1$ the primitive idempotents are $e=vv^T$, $v\in S^2$. We denote this cubic by $\Cartan d$ and call it \emph{Cartan's cubic}; the real case $d=1$, $n=5$, is Cartan's cubic $\Cartan1$. The algebras of Cartan's cubics are exceptional \cite[Chap.~6]{Nadirashvili-Tkachev-Vladuts}, \cite{Tkachev-book}.
\item The cubic $u=\det X$ on $M_3(\R{})\cong\R9$ with $\scal XY=\trace(XY^T)$. By the inequality between the arithmetic and geometric means of the singular values, $\det X\le|\det X|\le(|X|^2/3)^{3/2}$, with equality throughout if and only if $\sqrt3\,X\in SO(3)$. Hence $\max_Su=\frac1{3\sqrt3}$, and since every idempotent of a Hsiang algebra gives a maximum point (\Cref{the:equi}), $\Idm(\alg)=\frac12SO(3)$. This algebra is exceptional: it is an Einstein Hsiang algebra ($\Kill=4\Id$, see \Cref{tab:ratio}), and $n=9$ is not the dimension $3m_0$ of a mutant.
\end{enumerate}

\section{The Hsiang equation as an \texorpdfstring{$\infty$}{infinity}-eigenvalue problem}\label{sec:infty}

\begin{definition}\label{def:infeig}
A cubic form $u\not\equiv0$ on $\R{n}$ is called a \emph{harmonic $\infty$-eigencubic} if $\Delta u=0$ and $\Delta_\infty u=\lambda|x|^2u$ for some $\lambda\in\R{}$.
\end{definition}

The equation $\Delta_\infty u=\lambda|x|^2u$ is a nonlinear eigenvalue equation: $\Delta_\infty$ is homogeneous of degree three in $u$, and $\lambda$ is not an eigenvalue of a linear operator; in particular, it is different from the $\infty$-eigenvalue problem obtained as the limit of the $p$-Laplacian eigenvalue problems \cite{Juutinen-Lindqvist-Manfredi}.

\begin{proposition}\label{pro:infeig}
A cubic form $u\not\equiv0$ is a non-trivial Hsiang cubic if and only if it is a harmonic $\infty$-eigencubic, i.e.
\begin{equation}\label{infeig2}
\Delta u=0\qquad\text{and}\qquad\Delta_\infty u=\lambda|x|^2u\quad\text{for some }\lambda\in\R{}.
\end{equation}
In this case $\lambda=-\theta$ and
\begin{equation}\label{lambdaM}
\lambda=54\,M^2,\qquad M=\max_{S}u.
\end{equation}
\end{proposition}

\begin{proof}
If $u$ is a non-trivial Hsiang cubic, then $\Delta u=0$ by \Cref{the:exact}, and \eqref{hsiangeq} reduces to \eqref{infeig2} with $\lambda=-\theta$. Conversely, \eqref{infeig2} implies \eqref{hsiangeq} with $\theta=-\lambda$, and $u$ is non-trivial because $\Delta(a\scal qx^3)=6a|q|^2\scal qx\not\equiv0$.

To prove \eqref{lambdaM}, let $x_0\in S$ be a maximum point of $u|_S$. By Euler's identities, $\scal{Du(x)}{x}=3u(x)$ and $D^2u(x)x=2Du(x)$. At $x_0$ the gradient is normal to the sphere, so $Du(x_0)=3Mx_0$ and $D^2u(x_0)x_0=6Mx_0$. Therefore
$$
\Delta_\infty u(x_0)=9M^2\scal{D^2u(x_0)\,x_0}{x_0}=54M^3,
$$
while the right hand side of \eqref{infeig2} at $x_0$ equals $\lambda M$. Since $M>0$, \eqref{lambdaM} follows.
\end{proof}

\begin{remark}
The computation in the proof shows that a cubic form $u\not\equiv0$ is never $\infty$-harmonic: $\Delta_\infty u(x_0)=54M^3>0$ at a maximum point on the sphere. The same argument, applied at the maximum and minimum points on the sphere, shows that no homogeneous polynomial $u\not\equiv0$ of degree $k\ge2$ satisfies $\Delta_\infty u=0$; this yields the non-vanishing property of the gradient of real analytic $\infty$-harmonic functions \cite[Prop.~4.1]{Tkachev-laplace}, which goes back to Fuglede \cite{Fuglede-nonvanishing}, cf.\ Lewis \cite{Lewis-smoothness}. Thus \eqref{infeig2} is the simplest non-trivial $\infty$-Laplace equation which cubic forms can satisfy, and the Hsiang cubics are exactly its harmonic solutions. Formula \eqref{lambdaM} is the analytic form of the normalization $|c|^2=-\frac3{2\theta}$: the eigenvalue in \eqref{infeig2} is determined by the Chebyshev level $M$.
\end{remark}

\begin{remark}[$p$-Laplacians]
For a harmonic function with $Du\ne0$, $\Delta_pu=\operatorname{div}(|Du|^{p-2}Du)=(p-2)|Du|^{p-4}\Delta_\infty u$. Hence a non-trivial Hsiang cubic satisfies
$$
\Delta_pu=(p-2)\lambda\,|Du|^{p-4}|x|^2u\qquad\text{for every }p>1,
$$
an observation made in \cite[Sec.~4]{Tkachev-laplace}. This should be compared with the case $\lambda=0$: for $p\notin\{1,2\}$ the only cubic form satisfying $\Delta_pu=0$ in $\R{n}$ is $u\equiv0$ \cite{Tkachev-laplace}, which answers, for cubic forms, a question of J.~L.~Lewis \cite{Lewis-smoothness}; for homogeneous polynomials of degree four and five see Lewis and Vogel \cite{Lewis-Vogel}, and for the algebraic structure of homogeneous (quasiradial) solutions in the plane see \cite{Tkachev-quasiradial}. Thus the Hsiang cubics are the natural non-trivial cubic solutions of the nonlinear eigenvalue equations for the whole family of $p$-Laplacians. For $p=1$ one has $\Delta_1u:=|Du|\operatorname{div}\frac{Du}{|Du|}=-\lambda\frac{|x|^2u}{|Du|^2}$, which is Hsiang's original reading of \eqref{hsiangeq}: the mean curvature of the level set $\{u=s\}$ at $x$ equals $-\lambda s|x|^2/|Du|^3$, i.e.\ it is proportional to the level. For harmonic functions the $1$-Laplacian and the $\infty$-Laplacian coincide up to the factor $-|Du|^{-2}$; this is why the geometric (mean curvature) and the variational ($L^\infty$) readings of \eqref{hsiangeq} agree.
\end{remark}

\begin{remark}[Monotonicity along gradient lines]\label{rem:aronsson}
Aronsson's characterization of $\infty$-harmonic functions of class $C^2$ is that $|Du|$ is constant along the gradient curves $\dot\gamma=Du(\gamma)$ \cite{Aronsson-pde,Aronsson-Crandall-Juutinen}. For a Hsiang cubic,
$$
\frac{d}{dt}\,|Du(\gamma)|^2=2\Delta_\infty u(\gamma)=2\lambda\,|\gamma|^2u(\gamma),\qquad \frac{d}{dt}\,u(\gamma)=|Du(\gamma)|^2.
$$
Thus $|Du|$ is strictly increasing along gradient curves in $\{u>0\}$ and strictly decreasing in $\{u<0\}$, at a rate proportional to the level: the defect from Aronsson's condition is measured by $u$ itself. On the minimal cone $u^{-1}(0)$ (away from its singular set) the gradient is, to first order, of constant length along the flow.
\end{remark}

\begin{remark}[Kellogg's inequality and the Lipschitz problem]\label{rem:kellogg}
For a homogeneous polynomial $p$ of degree $k$, $\max_S|Dp|\le k\max_S|p|$ (Kellogg \cite{Kellogg-polynomials}); for $k=3$ it is a direct consequence of Banach's theorem \cite{Banach-homogene} (see, e.g., \cite[Prop.~1.44]{Dineen-book}, \cite{Agrachev-Kozhasov-Uschmajew}): every symmetric trilinear form $F$ on a Euclidean space satisfies
\begin{equation}\label{banach}
\sup_{|x|=|y|=|z|=1}|F(x,y,z)|=\sup_{|x|=1}|F(x,x,x)| .
\end{equation}
Indeed, if $p(x)=F(x,x,x)$, then $\scal{Dp(x)}y=3F(x,x,y)$ by the symmetry of $F$, hence
\begin{equation}\label{kellogg}
|Dp(x)|\le3\|p\|_\infty|x|^2,\qquad x\in\R n .
\end{equation}
Equality in $\max_S|Dp|\le3\|p\|_\infty$ holds for every cubic form, by the computation in the proof of \Cref{pro:infeig}. Since $Dp$ is homogeneous of degree $2$, the Lipschitz constant of a cubic form on the unit ball $B$ is therefore
$$
\operatorname{Lip}(p,B)=\max_S|Dp|=3\|p\|_\infty .
$$
Consequently the Chebyshev problem \eqref{kappa0} is the same as the problem of minimizing the Lipschitz constant on $B$ of a harmonic cubic with prescribed $L^2$-norm on the sphere. This is a genuine $L^\infty$-variational problem in the sense of Aronsson, restricted to a finite-dimensional class.
\end{remark}

\subsection{A first integral and the Bernstein--Szeg\H{o} inequality}\label{sec:firstint}
On the unit sphere, Euler's identity gives $|Du|^2=|\nabla_Su|^2+9u^2$, where $\nabla_S$ denotes the spherical gradient. Hence Kellogg's inequality of \Cref{rem:kellogg} can be written as
\begin{equation}\label{bsineq}
|\nabla_Su|^2+9u^2\le9\|u\|_\infty^2\qquad\text{on }S .
\end{equation}
For $n=2$, the restriction of a cubic form to the unit circle is a trigonometric polynomial $f(\varphi)$ of degree $3$, and \eqref{bsineq} becomes the Bernstein--Szeg\H{o} inequality $f'(\varphi)^2+9f(\varphi)^2\le9\|f\|_\infty^2$ \cite{Szego-Bernstein}. Equality at every point holds only for $f=M\cos3(\varphi-\varphi_0)$, that is, for the Chebyshev polynomial $T_3$ on the circle. Thus \eqref{bsineq} is a spherical form of the Bernstein--Szeg\H{o} inequality, and equality in it at every point is the \emph{eikonal equation} $|Du|^2=9M^2|x|^4$. For a Hsiang cubic the equality set of \eqref{bsineq} is large, and it is described by the following first integral.

\begin{proposition}\label{pro:firstint}
Let $u$ be a Hsiang cubic and $M=\|u\|_\infty$. Put
\begin{equation}\label{firstint}
Q(x)=9M^2|x|^4-|Du(x)|^2 .
\end{equation}
\begin{enumerate}
\item[(i)] $Q\ge0$ on $\R{n}$, and $Q$ is constant along every gradient curve $\dot\gamma=Du(\gamma)$.
\item[(ii)] $Q$ vanishes on the subspace $\R{}c\oplus\alg_c(-1)$ for every $c\in\Idm(\alg)$. In particular, equality holds in \eqref{bsineq} on the sphere $S_c=S\cap\bigl(\R{}c\oplus\alg_c(-1)\bigr)$ of dimension $n_{-1}$.
\end{enumerate}
\end{proposition}

\begin{proof}
(i) The inequality $Q\ge0$ is \eqref{bsineq}. Along a gradient curve, \Cref{rem:aronsson} gives $\frac{d}{dt}|Du(\gamma)|^2=2\lambda|\gamma|^2u(\gamma)$. On the other hand, $\frac{d}{dt}|\gamma|^2=2\scal{\gamma}{Du(\gamma)}=6u(\gamma)$ by Euler's identity, so $\frac{d}{dt}|\gamma|^4=12|\gamma|^2u(\gamma)$. Hence
$$
\frac{d}{dt}\,Q(\gamma)=(108M^2-2\lambda)\,|\gamma|^2u(\gamma)=0
$$
since $2\lambda=108M^2$ by \eqref{lambdaM}.

(ii) In the normalized scale $9M^2=\frac13$ and $|Du|^2=\frac14|x^2|^2$. Let $x=sc+y$ with $s\in\R{}$ and $y\in\alg_c(-1)$. By \Cref{the:peirce}, $cy=-y$ and $y^2=-\frac43|y|^2c$, so that
$$
x^2=\bigl(s^2-\tfrac43|y|^2\bigr)c-2s\,y .
$$
Since $|c|^2=\frac34$, this gives
$$
|x^2|^2=\tfrac34\bigl(s^2-\tfrac43|y|^2\bigr)^2+4s^2|y|^2=\tfrac43\bigl(\tfrac34s^2+|y|^2\bigr)^2=\tfrac43|x|^4,
$$
that is, $|Du(x)|^2=\frac13|x|^4=9M^2|x|^4$.
\end{proof}

Part (i) is the counterpart, for the $\infty$-eigenvalue equation \eqref{infeig2}, of Aronsson's characterization of $\infty$-harmonic functions. There $|Du|$ itself is constant along gradient curves. Here $|Du|^2$ grows, as noted in \Cref{rem:aronsson}, but the growth is compensated exactly by that of $9M^2|x|^4$. Since $Q$ is homogeneous, it follows that the set
$$
K(u)=\bigl\{x\in S:\ |\nabla_Su(x)|^2+9u(x)^2=9M^2\bigr\}
$$
where \eqref{bsineq} is sharp is invariant under the gradient flow of $u|_S$. For every cubic form this set contains the maximum and minimum points on $S$. For a Hsiang cubic, by (ii), it contains the spheres $S_c$, on which $u$ restricts to the Chebyshev form $M\,\Che_{3,n_{-1}+1}$ of \Cref{rem:aku}, and in particular the great circles of \Cref{the:equi}(iii), on which $u=M\cos3\varphi$. Thus the equality case of the one-dimensional Bernstein--Szeg\H{o} inequality, attained by $T_3$, propagates to spheres of dimension $n_{-1}$ through every idempotent. At the other extreme, $Q\equiv0$ means that $u$ is harmonic and satisfies the eikonal equation; these are exactly the Cartan--M\"unzner equations for isoparametric cubics, so $Q\equiv0$ if and only if $u$ is one of Cartan's cubics $\Cartan d$ \cite{Cartan-isoparametric,CecilRyan,Tkachev-generalization}. For them $K(u)=S$. More generally, by \cite{Tkachev-isoparametric}, the homogeneous polynomial solutions of the eikonal equation of any degree $m$ are exactly the compositions $T_k\circ G$ of a Chebyshev polynomial with a Cartan--M\"unzner polynomial $G$ of degree $m/k$. In this sense equality in the spherical Bernstein--Szeg\H{o} inequality characterizes the spherical Chebyshev polynomials in the most literal sense, and among harmonic cubics it singles out Cartan's cubics. This leads to the global result of the next section.

\section{A sharp lower bound: Cartan's cubics are global minimizers}\label{sec:cartanmin}

The following result does not use the structure theory of Hsiang algebras: it combines the Dirichlet energy of a harmonic cubic in the unit ball with Kellogg's inequality \eqref{kellogg}.

\begin{theorem}\label{the:cartanmin}
For every $n\ge2$ and every $0\ne p\in\hcub(\R{n})$,
\begin{equation}\label{cartanbound}
R(p)\ \ge\ \frac{n+4}{3}.
\end{equation}
Equality holds if and only if $|Dp|^2=9\|p\|_\infty^2|x|^4$ on $\R n$, i.e.\ if and only if, up to rotations and scaling, either $n=2$ and $p=\re z^3$, or $n=3d+2$ with $d\in\{1,2,4,8\}$ and $p$ is Cartan's cubic $\Cartan d$. In particular $\kappa_n\ge\frac{n+4}3$ for all $n$, with equality exactly for $n\in\{2,5,8,14,26\}$:
$$
\kappa_5=3,\qquad\kappa_8=4,\qquad\kappa_{14}=6,\qquad\kappa_{26}=10,
$$
and in these dimensions the minimum in \eqref{kappa0} is attained only on the orbit of Cartan's cubic.
\end{theorem}

\begin{proof}
Let $B$ be the unit ball, $|S|$ the area of $S$ and $M=\|p\|_\infty$. Since $p$ is harmonic and $\partial_rp=3p$ on $S$ by Euler's identity, Green's formula gives
$$
\int_B|Dp|^2\,dx=\int_Sp\,\partial_rp\,dS=3|S|\,\|p\|_2^2 .
$$
On the other hand, $|Dp(x)|\le3M|x|^2$ on $\R n$ by Kellogg's inequality \eqref{kellogg}. Since $\int_B|x|^4\,dx=|S|/(n+4)$,
$$
3|S|\,\|p\|_2^2\ \le\ 9M^2\int_B|x|^4\,dx=\frac{9M^2|S|}{n+4},
$$
which is \eqref{cartanbound}. Equality holds if and only if the continuous non-negative function $9M^2|x|^4-|Dp|^2$ vanishes on $B$, hence, by homogeneity, on $\R n$. By \cite[Thm.~1.4]{Tkachev-isoparametric} with $m=3$ (see also \cite{Cartan-isoparametric,Tkachev-generalization}), a cubic form satisfying the eikonal equation $|Dp|^2=9M^2|x|^4$ is, up to rotations and scaling, either one of Cartan's isoparametric cubics $\Cartan d$, $n=3d+2$, $d\in\{1,2,4,8\}$, or the cubic $x_n^3-3x_n(x_1^2+\dots+x_{n-1}^2)$, i.e.\ $T_3(\scal x{e_n})$ on the sphere. The latter has Laplacian $-6(n-2)x_n$, so it is harmonic only for $n=2$, where it is $\re z^3$. Conversely, these cubics are harmonic and satisfy the eikonal equation. The statement on $\kappa_n$ follows because the minimum in \eqref{kappa0} is attained.
\end{proof}

Cartan's cubics are thus the exact spherical counterparts of the Chebyshev polynomials. Equality in \eqref{bsineq} at every point of $S$ is the transnormality condition, and for an isoparametric cubic it integrates to
$$
u(x)=M\cos\bigl(3\operatorname{dist}_S(x,\hIdm)\bigr),\qquad x\in S,
$$
where $\operatorname{dist}_S$ is the geodesic distance and $\hIdm$ is the focal Veronese submanifold \cite{Munzner-I,Tkachev-isoparametric}. Thus $u=M\cos3\varphi$ along every great circle normal to $\hIdm$, not only along the circles of \Cref{the:equi}(iii); in Peirce terms, the normal space $\alg_c(-1)\oplus\alg_c(-\frac12)$ reduces to $\alg_c(-1)$, since $n_{-1/2}=0$ for Cartan's cubics.

\begin{remark}\label{rem:kelloggdefect}
The proof gives the identity
$$
R(p)=\frac{n+4}{3}\Bigl(1-\frac{\langle Q_p\rangle}{9M^2}\Bigr)^{-1},\qquad \langle Q_p\rangle=(n+4)\int_B\bigl(9M^2|x|^4-|Dp|^2\bigr)\frac{dx}{|S|},
$$
so that minimizing $R$ amounts to minimizing the mean defect in Kellogg's inequality. In dimensions $n\ne2,5,8,14,26$ no harmonic cubic is eikonal and the defect is bounded away from zero; for $n=3$ the value $\kappa_3=\frac{35}9$ of \Cref{conj:kappa} corresponds to the normalized mean defect $\langle Q_p\rangle/(9M^2)=\frac25$, attained by $x_1x_2x_3$ ($\Pol(1,2)$).
\end{remark}

For Hsiang cubics, \eqref{cartanbound} and \eqref{ratio} give an inequality for the Killing form which is purely algebraic in nature.

\begin{corollary}\label{cor:killingbound}
For every normalized Hsiang algebra $\alg$ on $\R n$,
$$
\trace \Kill\ \le\ \tfrac23\,n(n+2),
$$
with equality if and only if $\alg$ is the algebra of Cartan's cubic $\Cartan d$ (or of $\re z^3$ for $n=2$).
\end{corollary}

In this sense, in the dimensions where they exist, Cartan's cubics have ``the most multiplication'' among all Hsiang cubics.

\begin{proof}
By \eqref{ratio}, $R(u)=\frac{2N_n}{9\trace \Kill}$, and $R(u)\ge\frac{n+4}3$ is equivalent to $\trace \Kill\le\frac{2N_n}{3(n+4)}=\frac23n(n+2)$.
\end{proof}

\section{Equioscillation}\label{sec:equi}

The Chebyshev polynomial $T_k$ is singled out among polynomials of degree $k$ by three features: all its local extrema on $[-1,1]$ have the same absolute value; the points where this value is attained alternate in sign; and these alternation points, not the polynomial itself, are what enters the extremal characterization of $T_k$ and the associated quadrature formulas. In \Cref{sec:cartanmin} the most rigid case was settled: Cartan's cubics coincide with $M\cos3\varphi$ along every great circle normal to their maximum set. A general Hsiang cubic is not eikonal, and equality in \eqref{bsineq} holds only on the Chebyshev spheres $S\cap(\R{}c\oplus\alg_c(-1))$ of \Cref{sec:firstint}. Nevertheless, all three features survive, and they are encoded in the algebra: the extremal values are governed by the idempotents (\Cref{lem:critical}), the alternation set is the normalized idempotent variety $\pm\hIdm$, and the local behaviour of $u|_S$ near a maximum is described by the Peirce decomposition (\Cref{the:peirce}). This is the content of the following theorem, which is the basic input for the rest of the paper: the alternation set $\pm\hIdm$ carries the measures of \Cref{sec:kolm,sec:waring,sec:existence}, and the Peirce spectrum at its points controls the second variation in \Cref{sec:local}.

\begin{theorem}\label{the:equi}
Let $u$ be a Hsiang cubic and let $M$ be as in \eqref{Mnorm}.
\begin{enumerate}
\item[(i)] The critical values of $u|_S$ are contained in $\{M,0,-M\}$. Every local maximum of $u|_S$ is a global maximum, and
\begin{equation}\label{altset}
\{x\in S:\ |u(x)|=\|u\|_\infty\}=\hIdm\cup(-\hIdm),\qquad \hIdm=\{\hat c:\ c\in\Idm(\alg)\}.
\end{equation}
\item[(ii)] The Riemannian Hessian of $u|_S$ at $\hat c$ is
\begin{equation}\label{hessian}
\operatorname{Hess}(u|_S)(\hat c)=6M\bigl(L(c)-\tfrac12\bigr)\big|_{c^\perp},
\end{equation}
with eigenvalues $-9M$, $-6M$, $0$ on $\alg_c(-1)$, $\alg_c(-\frac12)$, $\alg_c(\frac12)$.
\item[(iii)] If $n_{-1}\ge1$, $\xi\in\alg_c(-1)$ and $|\xi|=1$, then
\begin{equation}\label{T3}
u(\cos\varphi\,\hat c+\sin\varphi\,\xi)=M\cos3\varphi=M\,T_3(\cos\varphi),\qquad\varphi\in\R{}.
\end{equation}
The six points $\varphi\in\frac\pi3\mathbb{Z}$ of alternation on this circle are $\pm\hat c_1,\pm\hat c_2,\pm\hat c_3$, where $c_1=c$ and $\{c_1,c_2,c_3\}$ is a \emph{triad} of idempotents, $c_1+c_2+c_3=0$.
\end{enumerate}
\end{theorem}

\begin{proof}
(i) By \Cref{lem:critical}, every critical point with $u(x)\ne0$ is $\pm\hat c$ with $c$ an idempotent, and $|u(\pm\hat c)|=\frac1{6|c|}$. Since all idempotents have the same length, these values are all equal to $M$. Let $x_0$ be a local maximum of $u|_S$. If $u(x_0)<0$, then $x_0=-\hat c$ is also a global minimum point, so $u|_S$ is constant near $x_0$ and $u=-M|x|^3$ on an open cone, which is impossible for a polynomial. If $u(x_0)=0$, then by homogeneity $u\le0$ on an open cone $\Gamma$ containing $x_0$, with $u(x_0)=0$; since $u$ is harmonic (\Cref{the:exact}), the strong maximum principle gives $u\equiv0$ on $\Gamma$, hence $u\equiv0$. Therefore every local maximum has the value $M$, and \eqref{altset} follows.

(ii) The Riemannian Hessian of the restriction of a function $f$ to the unit sphere at a critical point $x$ is $D^2f(x)-\scal{Df(x)}{x}\Id$ on $x^\perp$. With $f=u$, $x=\hat c$, we have $D^2u(\hat c)=L(\hat c)=L(c)/|c|$ and $\scal{Du(\hat c)}{\hat c}=3M=\frac1{2|c|}$, which gives \eqref{hessian}. The eigenvalues follow from \eqref{peirce}.

(iii) Put $\eta=|c|\xi$, so $|\eta|=|c|$. By \Cref{the:peirce}, $c\eta=-\eta$ and $\eta^2=-\frac43|\eta|^2c=-c$. Hence $\R{}c\oplus\R{}\eta$ is a subalgebra, and $(sc+t\eta)^2=(s^2-t^2)c-2st\,\eta$. Then
$$
u(sc+t\eta)=\tfrac16\scal{(s^2-t^2)c-2st\,\eta}{sc+t\eta}=\tfrac{|c|^2}6(s^3-3st^2).
$$
Substituting $s=\cos\varphi/|c|$, $t=\sin\varphi/|c|$ gives $u=\frac{1}{6|c|}\cos3\varphi$. Finally, the elements $c_2=-\frac12c+\frac{\sqrt3}2\eta$, $c_3=-\frac12c-\frac{\sqrt3}2\eta$ satisfy $c_2^2=c_2$, $c_3^2=c_3$ by the same multiplication rule, and $c_1+c_2+c_3=0$.
\end{proof}

\begin{remark}
Part (i) is the exact analogue of the defining property of $T_k$: all local extrema of $u|_S$ have the same absolute value. Part (ii) says that \emph{the Peirce decomposition of an idempotent is the spectral decomposition of the Hessian of the cubic at its maximum}, and that $u|_S$ is Morse--Bott along its maximum set, with nullity $n_{1/2}=\dim\Idm(\alg)$. Part (iii) shows that through every idempotent there passes an $(n_{-1}-1)$-parameter family of great circles on which $u$ is literally the Chebyshev polynomial $T_3$; the alternation points on these circles are hexagons of idempotents. For a generic cubic form none of these properties holds: its maximum on the sphere is attained at a single pair of points.
\end{remark}

\begin{remark}
Equioscillation in the sense of (i) does not need the full Hsiang identity: it only uses that all idempotents of $\alg(u)$ have the same length. Whether, conversely, a harmonic cubic whose critical values on the sphere are $\pm M$ and $0$ is necessarily a Hsiang cubic is open (\Cref{prob:converse}).
\end{remark}

\begin{remark}[Chebyshev forms]\label{rem:aku}
Following \cite{Agrachev-Kozhasov-Uschmajew}, the \emph{cubic Chebyshev form} in $m$ variables is
$$
\Che_{3,m}(s,y)=s^3-3s|y|^2,\qquad s\in\R{},\ y\in\R{m-1};
$$
on the unit sphere it coincides with $T_3(s)$, i.e.\ it restricts to $T_3(\cos\varphi)=\cos3\varphi$ on every great circle through $(1,0)$, and $\Che_{3,2}=\re(s+iy)^3$. The computation in the proof of \Cref{the:equi}(iii) holds verbatim for an arbitrary $y\in\alg_c(-1)$: by \Cref{the:peirce}, $(sc+y)^2=(s^2-\frac43|y|^2)c-2sy$, whence, with $s'=s|c|$,
\begin{equation}\label{chebform}
u(s'\hat c+y)=M\,\Che_{3,\,n_{-1}+1}(s',y),\qquad s'\in\R{},\ y\in\alg_c(-1).
\end{equation}
Thus \emph{the restriction of a Hsiang cubic to the subalgebra $\R{}c\oplus\alg_c(-1)$ is a cubic Chebyshev form}, and \eqref{T3} is its restriction to $2$-planes. For Cartan's cubic $\Cartan1$ one has $n_{-1}=2$, so its restriction to the $3$-dimensional subalgebra $\R{}c\oplus\alg_c(-1)$ is $M\,\Che_{3,3}$, which is exactly the ternary cubic of least deviation among \emph{all} ternary cubic forms found in \cite{Agrachev-Kozhasov-Uschmajew} (\Cref{pro:aku}). Note that $\Delta\Che_{3,m}=-6(m-2)s$, so $\Che_{3,m}$ is not harmonic for $m\ge3$, in accordance with the fact that $\R{}c\oplus\alg_c(-1)$ is a subalgebra but not an exact algebra. On the other hand, $|D\Che_{3,m}|^2=9|x|^4$, i.e.\ $\Che_{3,m}$ is a solution of the eikonal equation (for $m\ne2$ one of the reducible ones in the classification of \cite{Tkachev-generalization}, see also \cite{Tkachev-jordan}), and therefore
$$
\Delta_\infty\Che_{3,m}=\tfrac12\scal{D\Che_{3,m}}{D|D\Che_{3,m}|^2}=54\,|x|^2\,\Che_{3,m}.
$$
Thus the Chebyshev forms satisfy the $\infty$-eigenvalue equation of \Cref{def:infeig} with $\lambda=54M^2$, $M=\max_S\Che_{3,m}=1$, exactly as in \eqref{lambdaM}, but they are harmonic only for $m=2$. In this sense the extremal form of \cite{Agrachev-Kozhasov-Uschmajew} is a non-exact $\infty$-eigencubic, and the Hsiang cubics are its harmonic counterparts (\Cref{pro:infeig}).
\end{remark}

\section{The Chebyshev ratio of a Hsiang cubic}\label{sec:problem}

\subsection{The Fischer inner product}
For homogeneous polynomials $p,q$ of the same degree put $\Fs pq=p(\partial)q$. On monomials $\Fs{x^\alpha}{x^\beta}=\alpha!\,\delta_{\alpha\beta}$, so $\Fs{}{}$ is an $O(n)$-invariant inner product. We use two standard facts \cite{Stein-Weiss-book,Axler-Bourdon-Ramey}:
\begin{align}
\Fs{|x|^2r}{q}&=\Fs{r}{\Delta q},\label{adj}\\
\Fs{p}{\scal{x}{a}^k}&=k!\,p(a)\qquad(\deg p=k).\label{eval}
\end{align}

\begin{lemma}\label{lem:fischer}
If $p$ is a harmonic form of degree $k$ and $q$ is any form of degree $k$, then
$$
\int_Spq\,d\sigma=\frac{\Fs pq}{n(n+2)\cdots(n+2k-2)}.
$$
In particular, $\int_Spq\,d\sigma=\Fs pq/N_n$ for $p,q\in\hcub$, with $N_n$ as in \eqref{Nn}.
\end{lemma}

\begin{proof}
For harmonic $q$ both sides are $O(n)$-invariant inner products on the irreducible $O(n)$-module of harmonic forms of degree $k$, hence proportional; the constant is found on $p=q=x_1\cdots x_k$ if $k\le n$ (then $\Fs pp=1$ and $\int_Sx_1^2\cdots x_k^2\,d\sigma=1/(n(n+2)\cdots(n+2k-2))$), and on $p=q=\re z^3$ if $n=2$, $k=3$ (then $\Fs pp=24$ and $\int\cos^23\varphi=\frac12$). For general $q=h+|x|^2r$ with $h$ harmonic, both sides vanish on $|x|^2r$: the left one because $r|_S$ is a sum of spherical harmonics of lower degree, the right one by \eqref{adj}.
\end{proof}

\begin{lemma}[Funk--Hecke formula for cubics {\cite[Chap.~IV]{Stein-Weiss-book}}]\label{lem:zonal}
For $a\in S$ the harmonic part \eqref{harmdec} of $\scal xa^3$ is
\begin{equation}\label{zonal}
Z_a(x)=\scal xa^3-\frac{3}{n+2}|x|^2\scal xa,
\end{equation}
and
\begin{equation}\label{zonalquad}
\int_SpZ_a\,d\sigma=\frac6{N_n}\,p(a)\qquad\text{for all } p\in\hcub.
\end{equation}
Hence $K_a=\frac{N_n}6Z_a$ is the reproducing kernel of $\hcub\subset L^2(\sigma)$, and $K_a(a)=\dim\hcub=\frac{n(n-1)(n+4)}6$.
\end{lemma}

\begin{proof}
Since $\Delta\scal xa^3=6\scal xa$, \eqref{harmdec} gives $\ell=\frac{3}{n+2}\scal xa$. By \Cref{lem:fischer} and \eqref{eval},
$$
\int_S pZ_a\,d\sigma=\int_S p\,\scal xa^3\,d\sigma=\frac{\Fs{p}{\scal xa^3}}{N_n}=\frac{6\,p(a)}{N_n}.
$$
Finally $K_a(a)=\frac{N_n}6\cdot\frac{n-1}{n+2}=\frac{n(n-1)(n+4)}{6}$.
\end{proof}

\subsection{Trivial bounds}
Since $\sigma$ is a probability measure and $p$ is odd, $R(p)>1$. On the other hand, $|p(a)|=|\int pK_a\,d\sigma|\le\|p\|_2\|K_a\|_2$ and $\|K_a\|_2^2=K_a(a)$, so
\begin{equation}\label{trivial}
1<R(p)\le\dim\hcub,
\end{equation}
with equality on the right exactly for the zonal harmonics. The zonal harmonics are thus the ``peakiest'' harmonic cubics, and \eqref{kappa0} asks for the opposite extreme. For $n=2$ every harmonic cubic is a rotated multiple of $\re z^3$, so $R\equiv2=\kappa_2$.

\begin{theorem}\label{the:ratio}
For a normalized Hsiang cubic $u$ on $\R{n}$,
\begin{equation}\label{ratio}
\|u\|_2^2=\frac{\trace \Kill}{6N_n},\qquad R(u)=\frac{2N_n}{9\trace \Kill}.
\end{equation}
In the non-normalized scale $R(u)=-\theta N_n/(9\trace \Kill)$. If $\alg$ is an Einstein Hsiang algebra, $\Kill=\nu\Id$, then $R(u)=\frac{2(n+2)(n+4)}{9\nu}$.
\end{theorem}

\begin{proof}
Since $u$ is harmonic, \Cref{lem:fischer} gives $\|u\|_2^2=\Fs uu/N_n$. Writing $u=\frac16\sum \Tens_{ijk}x_ix_jx_k$ and differentiating, each of the $3!$ permutations contributes $\sum \Tens_{ijk}^2$, so $\Fs uu=\frac{6}{36}\|\Tens\|^2=\frac16\trace \Kill$ by \eqref{trA}. By \Cref{the:equi}, $\|u\|_\infty=M$ and $M^2=\frac1{27}$ (generally $M^2=-\theta/54$ by \eqref{lambdaM}).
\end{proof}

In a fixed dimension all Hsiang cubics have the same sup-norm $M$, so by \eqref{ratio} the flattest one is the one with the largest $\trace \Kill$. \Cref{tab:ratio} lists the examples of \Cref{sec:hsiang}; $\trace \Kill=6\Fs uu$ is computed from the monomial expansion.

\begin{table}[ht]\small
\renewcommand\arraystretch{1.25}
\begin{tabular}{l|c|c|c|c|c|l}
$u$ & $n$ & $\mathsf{pr}(\alg)$ & $\trace \Kill$ & $R(u)$ & $\dim\hcub$ & spectrum of $\Kill$\\\hline
$\re z^3$ & 2 & $(1,0,0)$ & $\frac{16}3$ & $2$ & $2$ & $\frac83,\frac83$\\
$\Pol(1,2)$, $x_1x_2x_3$ & 3 & $(0,2,0)$ & $6$ & $\frac{35}9$ & $7$ & $2,2,2$\\
$\Pol(2,2)$, Lawson & 4 & $(1,1,1)$ & $12$ & $\frac{32}9$ & $16$ & $2,2,4,4$\\
$\Pol(1,4)$ & 5 & $(0,3,1)$ & $12$ & $\frac{35}6$ & $30$ & $2,2,2,2,4$\\
Cartan $\Cartan1$ & 5 & $(2,0,2)$ & $\frac{70}3$ & $3$ & $30$ & $\frac{14}3$ ($\times5$)\\
$\det$ on $M_3(\R{})$ & 9 & $(0,5,3)$ & $36$ & $\frac{143}{18}$ & $156$ & $4$ ($\times9$)\\
$\Pol(2,8)$ & 10 & $(1,4,4)$ & $48$ & $\frac{70}9$ & $210$ & $4$ ($\times8$), $8$ ($\times2$)\\
$\mathrm{trip}(\operatorname{Im}\mathbb O)$ & 21 & $(4,5,11)$ & $252$ & $\frac{575}{54}$ & $1750$ & $12$ ($\times21$)
\end{tabular}
\medskip
\caption{Peirce profile \eqref{profile}, Killing trace and Chebyshev ratio of the Hsiang cubics of \Cref{sec:hsiang}; $\Pol(m_0,m_1)$ is the polar algebra of type $(m_0,m_1)$, see \Cref{def:polar}, and $\mathrm{trip}(\operatorname{Im}\mathbb O)$ is the $21$-dimensional tripling of \cite{Fox-Tkachev}, see \Cref{sec:peetre}.}\label{tab:ratio}
\end{table}

\noindent The column $\dim\hcub$ in \Cref{tab:ratio} is the value of $R$ at the zonal harmonics, cf.\ \eqref{trivial}; it shows how far the Hsiang cubics are from this extreme. In dimension five Cartan's cubic is much flatter than the cubic of $\Pol(1,4)$; we shall see in \Cref{sec:local} that the latter is not even a local minimizer.

\subsection{Comparison with all cubic forms}
For a cubic form $p=\sum_{|\alpha|=3}c_\alpha x^\alpha$ the \emph{Bombieri norm} is $\|p\|_B^2=\sum_{|\alpha|=3}\binom{3}{\alpha}^{-1}c_\alpha^2=\frac16\Fs pp$; equivalently, $\|p\|_B$ is the Frobenius norm of the symmetric $3$-tensor $S$ with $p(x)=S(x,x,x)$, while $\|p\|_\infty$ is its spectral norm. The quotient $\|p\|_\infty/\|p\|_B$ is the \emph{rank-one approximation ratio} of $p$, and its minimum over all cubic forms in $n$ variables measures how far symmetric tensors can be from rank one \cite{Agrachev-Kozhasov-Uschmajew}. If $p\in\hcub$, then by \Cref{lem:fischer}
\begin{equation}\label{bombieri}
R(p)=\frac{N_n}6\Bigl(\frac{\|p\|_\infty}{\|p\|_B}\Bigr)^2 ,
\end{equation}
so \eqref{kappa0} is the rank-one approximation ratio problem restricted to harmonic forms, i.e.\ to \emph{traceless} symmetric tensors.

\begin{remark}[Bombieri versus $L^2$]\label{rem:bombieri}
The proportionality \eqref{bombieri} holds only on $\hcub$. For a general cubic form write $p=h+|x|^2\ell$ as in \eqref{harmdec}. Since $\Fs{h}{|x|^2\ell}=\Fs{\Delta h}{\ell}=0$ and $\Fs{|x|^2\ell}{|x|^2\ell}=\Fs{\ell}{\Delta(|x|^2\ell)}=(2n+4)\Fs\ell\ell$ by \eqref{adj}, and since $h$ and $\ell$ are $L^2$-orthogonal on $S$,
$$
\|p\|_B^2=\frac{\Fs hh}{6}+\frac{(n+2)\Fs\ell\ell}{3},\qquad
\|p\|_2^2=\frac{\Fs hh}{N_n}+\frac{\Fs\ell\ell}{n}.
$$
Thus the two norms are proportional on a cubic form only if it is harmonic or of the form $|x|^2\ell$, and the problem of minimizing $\|p\|_\infty^2/\|p\|_2^2$ over \emph{all} cubic forms differs from that of \cite{Agrachev-Kozhasov-Uschmajew}. For instance, on the unit sphere $\Che_{3,3}=4x_1^3-3x_1|x|^2$ coincides with $T_3(x_1)$; here $\|\Che_{3,3}\|_B^2=\frac{32}5+\frac35=7$, while $\|\Che_{3,3}\|_2^2=\frac{17}{35}$.
\end{remark}

\begin{proposition}\label{pro:aku}
$\kappa_2=2$ and $\kappa_3\ge\frac52$.
\end{proposition}

\begin{proof}
By \cite[Theorems~1.1 and~1.2]{Agrachev-Kozhasov-Uschmajew}, every non-zero cubic form satisfies $\|p\|_\infty/\|p\|_B\ge\frac12$ if $n=2$ and $\|p\|_\infty/\|p\|_B\ge\frac1{\sqrt7}$ if $n=3$, with equality for $\re(x_1+ix_2)^3$ and for $\Che_{3,3}$, respectively. Restricting to harmonic forms and using \eqref{bombieri} with $N_2=48$ and $N_3=105$ gives $R(p)\ge\frac{48}{6}\cdot\frac14=2$ for $n=2$ and $R(p)\ge\frac{105}{6}\cdot\frac17=\frac52$ for $n=3$. For $n=2$ the bound is attained by the harmonic form $\re z^3$.
\end{proof}

For $n=2$ the extremal binary cubic of \cite{Agrachev-Kozhasov-Uschmajew} is harmonic, so the two problems have the same solution. For $n=3$ they differ: the extremal ternary cubic $\Che_{3,3}$ is not harmonic, and \Cref{conj:kappa} asserts that the harmonic constraint raises the extremal value from $\frac52$ to $\frac{35}9$, the extremal form changing from $\Che_{3,3}$ to $x_1x_2x_3$ ($\Pol(1,2)$). By \Cref{rem:aku}, both extremal forms come from Hsiang algebras: $x_1x_2x_3$ is a Hsiang cubic, and $\Che_{3,3}$ is the restriction of Cartan's cubic $\Cartan1$ to the subalgebra $\R{}c\oplus\alg_c(-1)$.

\begin{remark}[Chebyshev forms are not extremal for $n\ge5$]\label{rem:akuconj}
In \cite[Conjecture~1]{Agrachev-Kozhasov-Uschmajew} it is conjectured that the Chebyshev forms are extremal in every dimension; for cubics this reads
$$
\frac{\|p\|_\infty^2}{\|p\|_B^2}\ \ge\ \frac{\|\Che_{3,n}\|_\infty^2}{\|\Che_{3,n}\|_B^2}=\frac1{3n-2}\qquad\text{for all cubic forms }p\ne0\text{ on }\R{n},
$$
the value on the right following from $\|\Che_{3,n}\|_\infty=1$ and $\|\Che_{3,n}\|_B^2=1+3(n-1)$. For a Hsiang cubic, \eqref{bombieri} and \Cref{the:ratio} give
\begin{equation}\label{hsiangbombieri}
\frac{\|u\|_\infty^2}{\|u\|_B^2}=\frac{6R(u)}{N_n}=\frac{4}{3\trace \Kill}.
\end{equation}
For Cartan's cubic $\Cartan1$ ($n=5$, $\trace \Kill=\frac{70}3$) this equals $\frac{2}{35}<\frac1{13}$, so the conjecture fails for cubic forms on $\R5$. The same happens for the other Hsiang cubics with a large Killing trace in \Cref{tab:ratio}: $\det X$ gives $\frac1{27}<\frac1{25}$ ($n=9$), $\Pol(2,8)$ gives $\frac1{36}<\frac1{28}$ ($n=10$), and $\mathrm{trip}(\operatorname{Im}\mathbb O)$ gives $\frac1{189}<\frac1{61}$ ($n=21$). In $\R4$, on the other hand, the Lawson cubic ($\Pol(2,2)$) gives $\frac19>\frac1{10}$, and the case $n=4$ remains open. In fact the conjecture fails for every $n\ge5$. If $p_1,p_2$ are cubic forms on orthogonal subspaces with $\|p_1\|_\infty=\|p_2\|_\infty$, then $\|p_1+p_2\|_\infty=\|p_1\|_\infty$ (since $s^3+t^3\le(s^2+t^2)^{3/2}$ for $s,t\ge0$) and the Bombieri norms add in squares, so the reciprocal ratios $\|p\|_B^2/\|p\|_\infty^2$ add. By \eqref{hsiangbombieri} and \Cref{cor:killingbound} this reciprocal equals $\frac12n(n+2)$ for Cartan's cubic on $\R n$ and $4$ for $\re z^3$, and comparing with $3n-2$ one checks that $\Cartan1$ (for $n=5,6$), $\Cartan1\oplus\re z^3$ ($n=7$), $\Cartan2$ ($8\le n\le13$), $\Cartan4$ ($14\le n\le37$) and orthogonal sums of $k$ copies of $\Cartan8$ ($26k\le n\le121k$) beat the Chebyshev form. Thus Hsiang algebras with much multiplication are flatter than the Chebyshev forms, even among all cubic forms.
\end{remark}

\begin{remark}[Heuristics]\label{rem:heuristics}
Several facts proved in this paper point towards Hsiang cubics as the extremal harmonic cubics.
\begin{enumerate}
\item \emph{Chebyshev principle.} A minimizer of \eqref{kappa0} is stationary, and by \Cref{the:kolm} and the remark following it, a stationary harmonic cubic which is not zonal must equioscillate on a large set: it must be a positive average of zonal harmonics with poles in its own extremal set. Hsiang cubics are the harmonic cubics in which this happens in the most rigid way: all positive critical values on the sphere coincide, and the extremal set is the whole idempotent variety (\Cref{the:equi}). For $n=2$ the analogy is complete, since $\kappa_2=2$ is attained exactly by $\re z^3$.
\item \emph{Flatness is measured by the multiplication.} Among Hsiang cubics of a given dimension, \eqref{ratio} shows that $R$ decreases as $\trace \Kill=\|\Tens\|^2$ increases. In $\R5$ Cartan's cubic has almost twice the value of $\trace \Kill$ of $\Pol(1,4)$, and \Cref{pro:polarsaddle} shows that $R$ indeed decreases from the cubic of $\Pol(1,4)$ in the direction of the Lawson component of Cartan's cubic, cf.\ \eqref{cartansplit}.
\item \emph{Rigidity.} For $x_1x_2x_3$ ($\Pol(1,2)$), the Lawson cubic ($\Pol(2,2)$) and Cartan's cubic $\Cartan1$ every direction which vanishes on the idempotent variety is an infinitesimal rotation, $D=o(u)$; for the first two this makes them strict local minimizers (\Cref{cor:strict}), and for $\Cartan1$ it is consistent with its global minimality (\Cref{rem:cartanlocal}), whereas the cubic of $\Pol(1,4)$ is not a local minimizer (\Cref{pro:polarsaddle}). Moreover, the alternation sets of $x_1x_2x_3$ and of Cartan's cubic $\Cartan1$ carry spherical $3$-designs (\Cref{cor:design}), the strongest form of equidistribution compatible with the degree.
\end{enumerate}
\end{remark}

\begin{conjecture}\label{conj:kappa}
$\kappa_3=\frac{35}9$ and $\kappa_4=\frac{32}9$, and the minimum in \eqref{kappa0} is attained exactly on the $O(n)$-orbits (up to scale) of $x_1x_2x_3$ ($\Pol(1,2)$) and of the Lawson cubic ($\Pol(2,2)$), respectively. (The analogous statement $\kappa_5=3$ for Cartan's cubic $\Cartan1$ is \Cref{the:cartanmin}.)
\end{conjecture}

\subsection{Block-trilinear cubics and trilinear forms}\label{sec:peetre}
\emph{The analytic side.} The trilinear analogue of the comparison in \Cref{pro:aku} goes back to Cobos, K\"uhn and Peetre \cite{Cobos-Kuhn-Peetre-Duke,Cobos-Kuhn-Peetre-Sp}, who developed Schatten--von Neumann classes of multilinear forms on Hilbert spaces. For a real trilinear form $T$ on $(\R{k})^3$ let $\|T\|=\sup\{|T(x,y,z)|:|x|=|y|=|z|=1\}$ and let $\|T\|_2$ be its Hilbert--Schmidt norm, and let $d(k)$ be the best constant in
\begin{equation}\label{ckp}
\|T\|_2^2\le d(k)\,\|T\|^2 .
\end{equation}
For bilinear forms the analogous constant is $k$, by the singular value decomposition; for trilinear forms there is no such decomposition, and $d(k)$ is not known in general. One always has $d(k)\le k^2$, and $d(k)=k^2$ for $k=1,2,4,8$, the extremal form being $\re(xyz)$ on the normed division algebra $\mathbb K_k$, whereas $7\le d(3)\le7.36$ \cite{Kuhn-Peetre-embedding}.

\emph{The algebraic side.} A trilinear form on $(\R{k})^3$ is the same thing as a cubic form on $\R{3k}$ of a special kind.

\begin{definition}\label{def:blocktri}
A cubic form $u$ on $\R{n}$ is called \emph{block-trilinear} if its metrized algebra admits an orthogonal decomposition
\begin{equation}\label{blocktri}
\alg=\alg_1\oplus\alg_2\oplus\alg_3,\qquad \alg_i\alg_i=\{0\},\qquad \alg_i\alg_j\subset\alg_k\quad\text{for }\{i,j,k\}=\{1,2,3\}.
\end{equation}
\end{definition}

Since $\scal{xy}z$ is the full polarization of $6u$, \eqref{blocktri} holds if and only if, in coordinates $x=(x_1,x_2,x_3)$ adapted to the decomposition, $u(x)=T(x_1,x_2,x_3)$ for a trilinear form $T$ on $\alg_1\times\alg_2\times\alg_3$; then $\scal{x_iy_j}{z_k}=T(\dots)$ with the arguments in the appropriate order. A block-trilinear cubic has degree at most one in each block of variables and is therefore harmonic, but it need not be a Hsiang cubic. Two families of Hsiang cubics are block-trilinear, and both are governed by Hurwitz's theorem.
\begin{enumerate}
\item[(a)] \emph{Composition: the mutants.} For $T(x,y,z)=\re\bigl((xy)z\bigr)$ on a normed division algebra $\mathbb K_k$, $k\in\{1,2,4,8\}$, the cubic $u$ is the mutant of \Cref{sec:polar}. Here every slice $T(x,\cdot,\cdot)$ is $|x|$ times an orthogonal map, i.e.\ $L(x_i)^2=|x_i|^2$ on $\alg_j\oplus\alg_k$: in the language of \cite{Tkachev-book}, $(\alg_1,\alg_2,\alg_3)$ is a \emph{triality system}. Accordingly $\|T\|=1$ and $\|T\|_2^2=k^2$, so these are exactly the extremal forms of Cobos, K\"uhn and Peetre.
\item[(b)] \emph{Vector products: the triplings.} Let $\times$ be a vector product on $\R{k}$, $k\in\{3,7\}$. The tripling construction of \cite{Fox-Tkachev} associates with the vector product algebra $(\R{k},\times)$ the algebra $\R{k}\otimes\R3$, the tensor product with the vector product algebra $\R3$; for $x=\sum_ix_i\otimes e_i$ one finds $x^2=2\sum_{\mathrm{cyclic}}(x_i\times x_j)\otimes e_k$, and its cubic form is
$$
u(x_1,x_2,x_3)=\scal{x_1\times x_2}{x_3},
$$
block-trilinear with an alternating form $T$. For $k=3$ this is $u=\det[x_1\,x_2\,x_3]$, i.e.\ $\det X$ on $M_3(\R{})$ of \Cref{sec:hsiang}; for $k=7$ it is the $21$-dimensional exceptional Hsiang algebra of \cite{Fox-Tkachev}. Here the slice $T(x,\cdot,\cdot)$ is skew-symmetric with kernel $\R{}x$, so it is never orthogonal: $\|T\|=1$ but $\|T\|_2^2=k(k-1)$.
\end{enumerate}
Thus the analytic quantity $\|T\|_2^2/\|T\|^2$ measures how far the decomposition \eqref{blocktri} is from a triality system, and the comparison between (a) and (b) is the comparison between compositions and vector products in Hurwitz's theorem. The next proposition translates \eqref{ckp} into the Chebyshev ratio.

\begin{proposition}\label{pro:peetre}
Let $u(x_1,x_2,x_3)=T(x_1,x_2,x_3)$ be a non-zero block-trilinear cubic on $\R{n}=(\R{k})^3$, so that $n=3k$ and $N_n=3k(3k+2)(3k+4)$. Then $u\in\hcub$ and
\begin{equation}\label{peetre}
R(u)=\frac{N_n}{27}\,\frac{\|T\|^2}{\|T\|_2^2}\ \ge\ \frac{N_n}{27\,k^2}.
\end{equation}
Equality holds if and only if $k\in\{1,2,4,8\}$ and, up to orthogonal changes of the variables $x_1,x_2,x_3$ and scaling, $u=\re\bigl((x_1x_2)x_3\bigr)$ on $\mathbb K_k^3$, i.e.\ $u$ is a mutant Hsiang cubic.
\end{proposition}

\begin{proof}
$u$ is harmonic by the remark after \Cref{def:blocktri}. Since $|T(x_1,x_2,x_3)|\le\|T\|\,|x_1||x_2||x_3|$ with equality for suitable unit vectors, and on $S^{n-1}$, by the inequality between the arithmetic and geometric means,
$$
\bigl(|x_1|^2|x_2|^2|x_3|^2\bigr)^{1/3}\le\frac{|x_1|^2+|x_2|^2+|x_3|^2}{3}=\frac13 ,
$$
i.e.\ $|x_1||x_2||x_3|\le3^{-3/2}$, with equality iff $|x_i|^2=\frac13$, we get $\|u\|_\infty=\|T\|/(3\sqrt3)$. The symmetric tensor of $u$ has six non-zero blocks, each equal to $\frac16T$ up to a permutation of the arguments, so $\|u\|_B^2=\frac16\|T\|_2^2$. Now \eqref{bombieri} gives the identity in \eqref{peetre}. For an orthonormal basis $e_1,\dots,e_k$ of $\R{k}$ let $T_i=T(e_i,\cdot,\cdot)$, regarded as a $k\times k$ matrix. Then
$$
\|T\|_2^2=\sum_{i=1}^k\|T_i\|_2^2\le k\sum_{i=1}^k\|T_i\|_{\mathrm{op}}^2\le k^2\|T\|^2,
$$
which is the inequality. Equality forces $T_i=\|T\|\,O_i$ with $O_i$ orthogonal, for every orthonormal basis, hence $T(x,\cdot,\cdot)$ is $\|T\|\,|x|$ times an orthogonal matrix for every $x$. Thus $(x,y)\mapsto\mu(x,y)$, $\scal{\mu(x,y)}z=T(x,y,z)/\|T\|$, is a bilinear map with $|\mu(x,y)|=|x||y|$. Fix unit vectors $a,b$; the maps $y\mapsto\mu(a,y)$ and $x\mapsto\mu(x,b)$ are orthogonal, and $x\circ y=\mu\bigl(\mu(\cdot,b)^{-1}x,\mu(a,\cdot)^{-1}y\bigr)$ is a composition with unit $\mu(a,b)$, $|x\circ y|=|x||y|$. By Hurwitz's theorem on composition algebras with unit \cite{Hurwitz1898}, see also \cite[Chap.~1]{Shapiro-compositions}, $k\in\{1,2,4,8\}$ and $(\R k,\circ)$ is isometrically isomorphic to $\mathbb K_k$; hence $\mu$ is isotopic to the multiplication of the normed division algebra $\mathbb K_k$, i.e.\ $\mu(x,y)=C\bigl((Ax)(By)\bigr)$ with orthogonal $A,B,C$, so that $u$ is congruent to $\scal{x_1x_2}{x_3}=\re\bigl((x_1x_2)\bar x_3\bigr)$; replacing $x_3$ by $\bar x_3$ gives the mutant cubic of \Cref{sec:polar}. Conversely, for the mutant cubic $\|T\|=1$ and $\|T\|_2^2=k^2$.
\end{proof}

\begin{corollary}\label{cor:mutantmin}
Let $k\in\{1,2,4,8\}$ and $n=3k$. On the subspace of block-trilinear cubics in $\hcub$, the minimum of $R$ equals $N_n/(27k^2)$ and is attained exactly at the mutant Hsiang cubic, up to orthogonal changes of the blocks and scaling. For $k=1$ this is the value $\frac{35}9$ of $x_1x_2x_3$ ($\Pol(1,2)$), and in general it agrees with \Cref{the:clifford} for $(m_0,m_1)=(k,2k)$.
\end{corollary}

\begin{remark}\label{rem:tripling}
For the triplings (b), \eqref{peetre} gives $R(u)=\frac{N_n}{27k(k-1)}$: for $k=3$ this is $\frac{9\cdot11\cdot13}{27\cdot6}=\frac{143}{18}$, in agreement with \Cref{tab:ratio}, and for $k=7$ it is $\frac{21\cdot23\cdot25}{27\cdot42}=\frac{575}{54}$ (\Cref{tab:ratio}; this algebra is an Einstein Hsiang algebra with $\Kill=12\Id$), compared with the lower bound $\frac{575}{63}$ in \eqref{peetre}, which is not attained since $7\notin\{1,2,4,8\}$. For $k=3$ the tripling $\det X$ is not extremal for \eqref{ckp} either, since $\|T\|_2^2=6<7\le d(3)$; the extremal trilinear forms on $(\R3)^3$ are not known.
\end{remark}

\section{The Kolmogorov criterion}\label{sec:kolm}

\subsection{First variation}
For a continuous function $p$ on $S$ let $E(p)=\{x\in S:\ |p(x)|=\|p\|_\infty\}$. The one-sided derivative of the uniform norm is given by Danskin's theorem on max-functions \cite[Chap.~III, Theorem~I]{Danskin}, applied to $s\,(p+\varepsilon h)(x)$ on the compact set $S\times\{\pm1\}$ (see also \cite{Singer-bestapproximation}, \cite[Thm.~23.4]{Rockafellar}): for every continuous $h$,
\begin{equation}\label{danskingen}
\frac{d}{d\varepsilon}\Big|_{\varepsilon=0^+}\|p+\varepsilon h\|_\infty=\max_{x\in E(p)}\sgn p(x)\,h(x).
\end{equation}
For $p=u$ we have $E(u)=\pm\hIdm$ by \eqref{altset}, and since $h\in\hcub$ is odd, the contributions of $\hat c$ and $-\hat c$ coincide:
\begin{equation}\label{danskin}
\frac{d}{d\varepsilon}\Big|_{0^+}\|u+\varepsilon h\|_\infty=\max_{c\in\Idm(\alg)}h(\hat c).
\end{equation}
If $h\perp u$ in $L^2(\sigma)$, then $\|u+\varepsilon h\|_2^2=\|u\|_2^2+\varepsilon^2\|h\|_2^2$, so that
\begin{equation}\label{dR}
\frac{d}{d\varepsilon}\Big|_{0^+}R(u+\varepsilon h)=\frac{2M}{\|u\|_2^2}\,\max_{c\in\Idm(\alg)}h(\hat c).
\end{equation}

\begin{definition}
A non-zero $p\in\hcub$ is \emph{stationary} for \eqref{kappa0} if no direction decreases $R$ to first order, i.e.\ if the one-sided derivative of $R(p+\varepsilon h)$ at $\varepsilon=0^+$ is non-negative for every $h\in\hcub$.
\end{definition}

Since $R$ is invariant under scaling, the component of $h$ along $p$ does not affect the first-order behaviour: $R\bigl(p+\varepsilon(\alpha p+h')\bigr)=R\bigl(p+\tfrac{\varepsilon}{1+\varepsilon\alpha}h'\bigr)$. Hence it suffices to consider $h\perp p$. Since the factor $2M/\|u\|_2^2$ in \eqref{dR} is positive, the derivative has the sign of $\max_ch(\hat c)$, and \eqref{dR} shows that \emph{a Hsiang cubic $u$ is stationary if and only if}
\begin{equation}\label{stat}
\max_{c\in\Idm(\alg)}h(\hat c)\ \ge\ 0\qquad\text{for every }h\in\hcub,\ h\perp u.
\end{equation}

Every local minimizer of $R$ is stationary, but the converse fails badly.

\begin{proposition}\label{pro:zonalstat}
Every zonal harmonic $Z_a$, $a\in S$, is stationary.
\end{proposition}

\begin{proof}
For $n=2$ every harmonic cubic is a rotated multiple of $\re z^3$, so $R$ is constant on $\hcub\setminus\{0\}$. Let $n\ge3$. We have $Z_a(x)=G(\scal xa)$ with $G(t)=t^3-\frac{3}{n+2}t$. The interior critical values of $G$ on $[-1,1]$ have modulus $2(n+2)^{-3/2}<\frac{n-1}{n+2}=G(1)$, so $|G|$ attains its maximum only at $t=\pm1$, i.e.\ $E(Z_a)=\{\pm a\}$. Since $h\in\hcub$ is odd, \eqref{danskingen} gives
$$
\frac{d}{d\varepsilon}\Big|_{\varepsilon=0^+}\|Z_a+\varepsilon h\|_\infty=h(a).
$$
If $h\perp Z_a$, then $h(a)=\int_S hK_a\,d\sigma=0$ by \Cref{lem:zonal}, and $\|Z_a+\varepsilon h\|_2^2=\|Z_a\|_2^2+\varepsilon^2\|h\|_2^2$. Hence the one-sided derivative of $R(Z_a+\varepsilon h)$ vanishes for every $h\perp Z_a$, and, by scale invariance, $Z_a$ is stationary.
\end{proof}

Since the zonal harmonics are the global maximizers of $R$ by \eqref{trivial}, stationarity cannot distinguish minimizers from maximizers; this is why the finer analysis of \Cref{sec:local} is needed.

\begin{theorem}[Kolmogorov criterion]\label{the:kolm}
For a Hsiang cubic $u$ the following are equivalent:
\begin{enumerate}
\item[(i)] $u$ is stationary;
\item[(ii)] there is a probability measure $\mu_0$ on $\Idm(\alg)$ with $\int h(\hat c)\,d\mu_0(c)=0$ for all $h\in\hcub$, $h\perp u$;
\item[(iii)] there is a finite positive measure $\mu$ on $\Idm(\alg)$ such that
\begin{equation}\label{kolmrep}
u(x)=\int_{\Idm(\alg)}Z_{\hat c}(x)\,d\mu(c)=\int_{\Idm(\alg)}\Bigl(\scal{x}{\hat c}^3-\frac{3}{n+2}|x|^2\scal x{\hat c}\Bigr)d\mu(c);
\end{equation}
\item[(iv)] there is a finite positive measure $\mu$ on $\Idm(\alg)$ such that
\begin{equation}\label{cubature}
\int_Sh\,u\,d\sigma=\frac6{N_n}\int_{\Idm(\alg)}h(\hat c)\,d\mu(c)\qquad\text{for all }h\in\hcub.
\end{equation}
\end{enumerate}
The measures in (iii) and (iv) are the same, and may be chosen with at most $\dim\hcub$ atoms. Their total mass is
\begin{equation}\label{mass}
\mu(\Idm(\alg))=\frac{|c|}6\trace \Kill=\frac{\sqrt3}{12}\trace \Kill.
\end{equation}
\end{theorem}

\begin{definition}\label{def:harmrep}
A finite positive measure $\mu$ on $\Idm(\alg)$ satisfying \eqref{kolmrep}, equivalently \eqref{cubature}, is called a \emph{Kolmogorov measure} of $u$. In the language of the truncated moment problem, a Kolmogorov measure represents the functional $h\mapsto\int_Sh\,u\,d\sigma$ on $\hcub$, i.e.\ it represents $u$ modulo the trace part $|x|^2\mathscr H_1$ in \eqref{harmdec}.
\end{definition}

\begin{proof}
(i)$\Leftrightarrow$(ii). Let $V=u^\perp\subset\hcub$ and let $\mathcal K\subset V^*$ be the image of the compact set $\Idm(\alg)$ under $c\mapsto(h\mapsto h(\hat c))$. Since $\hcub$ is finite dimensional, $\mathcal K$ is a compact subset of the finite-dimensional space $V^*$, and its convex hull $\conv\mathcal K$ is compact by Carath\'eodory's theorem. By the strict separation theorem for compact convex sets in finite dimensions, either there is $h\in V$ with $\ell(h)<0$ for all $\ell\in\mathcal K$, i.e.\ \eqref{stat} fails, or $0\in\conv\mathcal K$, which is (ii). By Carath\'eodory's theorem on convex hulls, $0$ is then a convex combination of at most $\dim V+1=\dim\hcub$ points of $\mathcal K$; this is the elementary part of Tchakaloff's theorem on positive cubature formulas \cite{Putinar-Tchakaloff,Bayer-Teichmann}.

(ii)$\Rightarrow$(iii). Let $\mu_0$ be as in (ii) and put
$$
p=\int_{\Idm(\alg)}Z_{\hat c}\,d\mu_0(c).
$$
Since $c\mapsto Z_{\hat c}$ is a continuous map into the finite-dimensional space $\hcub$, $p$ is a well-defined element of $\hcub$. For every $h\in\hcub$, by Fubini's theorem and \eqref{zonalquad},
\begin{equation}\label{kolmpair}
\int_Sh\,p\,d\sigma=\int_{\Idm(\alg)}\Bigl(\int_ShZ_{\hat c}\,d\sigma\Bigr)d\mu_0(c)=\frac6{N_n}\int_{\Idm(\alg)}h(\hat c)\,d\mu_0(c).
\end{equation}
By (ii) the right hand side vanishes for $h\in V$, so $p\perp V$ and therefore $p=\gamma u$ for some $\gamma\in\R{}$. Taking $h=u$ in \eqref{kolmpair} and using $u(\hat c)=M$ on $\Idm(\alg)$ (\Cref{the:equi}),
$$
\gamma\|u\|_2^2=\frac6{N_n}\int u(\hat c)\,d\mu_0=\frac{6M}{N_n}\,\mu_0(\Idm(\alg))>0 .
$$
Hence $\gamma>0$, and $\mu=\mu_0/\gamma$ satisfies \eqref{kolmrep}. The second expression in \eqref{kolmrep} is just the formula \eqref{zonal} for $Z_{\hat c}$.

(iii)$\Leftrightarrow$(iv). If $\mu$ satisfies \eqref{kolmrep}, then \eqref{kolmpair} with $\mu$ in place of $\mu_0$ and $p=u$ is exactly \eqref{cubature}. Conversely, if $\mu$ satisfies \eqref{cubature}, then $p=\int Z_{\hat c}\,d\mu\in\hcub$ satisfies $\int_Sh\,p\,d\sigma=\int_Sh\,u\,d\sigma$ for all $h\in\hcub$ by \eqref{kolmpair}; since $p-u\in\hcub$, this forces $p=u$, which is \eqref{kolmrep}. In particular the measures in (iii) and (iv) are the same, and if $\mu_0$ in (ii) has at most $\dim\hcub$ atoms, so does $\mu$.

(iv)$\Rightarrow$(i). Taking $h=u$ in \eqref{cubature} gives $\|u\|_2^2=\frac{6M}{N_n}\mu(\Idm(\alg))$, so $\mu\ne0$. If $h\in\hcub$ and $h\perp u$, then \eqref{cubature} gives $\int h(\hat c)\,d\mu(c)=0$; since $\mu$ is a non-zero positive measure, $h(\hat c)\ge0$ for some $c\in\Idm(\alg)$, i.e.\ \eqref{stat} holds.

Finally, the identity $\|u\|_2^2=\frac{6M}{N_n}\mu(\Idm(\alg))$ together with $\|u\|_2^2=\frac{\trace \Kill}{6N_n}$ from \eqref{ratio} and $6M=1/|c|$ gives \eqref{mass}.
\end{proof}

\begin{remark}[why stationarity is special]
For a generic $p\in\hcub$, $E(p)$ is a single pair $\pm x_0$, and the analogue of (iii) forces $p$ to be a multiple of $Z_{x_0}$. Thus a stationary harmonic cubic which is not zonal must have a large extremal set, large enough to reproduce $p$ as a positive average of zonal harmonics. In the classical case it is the $k+1$ alternation points of $T_k$ that play this role. For Hsiang cubics the extremal set is the whole idempotent variety, which explains the analogy.
\end{remark}

\begin{remark}[Quadrature identity]
We say that $u$ \emph{admits an idempotent quadrature} if \eqref{cubature} holds; thus a Hsiang cubic is stationary exactly when it admits an idempotent quadrature, in analogy with quadrature domains, where $\int_\Omega h\,dA$ is reproduced by finitely many point evaluations for all integrable harmonic $h$. Identity \eqref{cubature} is a \emph{quadrature identity with the weight $u$} (we use the word in the sense of quadrature domains; in numerical analysis multivariate formulas of this kind are called cubature formulas): the integral of every harmonic cubic against $u\,d\sigma$ is computed exactly from its values at the idempotents. It is the analogue of the positive quadrature formulas carried by the alternation points of Chebyshev polynomials. We return to this point of view in \Cref{sec:open}.
\end{remark}

\section{Positive Waring decompositions and moments}\label{sec:waring}

Recall the distinction between the two kinds of measures: a Kolmogorov measure reproduces $u$ as an average of zonal harmonics, i.e.\ it reproduces only the harmonic part of the cubic $\int\scal x{\hat c}^3d\mu$, whereas a representing measure reproduces this cubic itself; the second notion is stronger.
Let $\mu$ be as in \eqref{kolmrep} and put $m=\int\hat c\,d\mu(c)$. Then \eqref{kolmrep} reads
\begin{equation}\label{waringtrace}
u(x)=\int_{\Idm(\alg)}\scal x{\hat c}^3\,d\mu(c)-\frac{3}{n+2}|x|^2\scal xm.
\end{equation}
This is a \emph{Waring decomposition} of $u$, a sum of cubes of linear forms corrected by a trace part in the sense of \eqref{harmdec}, with the special feature that all nodes lie in the \emph{maximum set} $\hIdm$ and all weights are positive. (Over $\R{}$ signs can always be absorbed into the nodes, since $-\scal xa^3=\scal x{-a}^3$; the point is that only maximum points, never minimum points, are used.) For Waring decompositions and their relation to cubature and designs see \cite{Reznick-memoir}.

\begin{definition}
Following the terminology of the truncated moment problem \cite{Curto-Fialkow-flat,Schmudgen-book,Putinar-positive}, a positive measure $\mu$ on $\Idm(\alg)$ is called a \emph{representing measure} of $u$ (or of the tensor $\frac16\Tens$) if
\begin{equation}\label{repmeas}
u(x)=\int_{\Idm(\alg)}\scal x{\hat c}^3\,d\mu(c),\qquad\text{equivalently}\qquad L(x)=6\int_{\Idm(\alg)}\scal x{\hat c}\ \hat c\otimes\hat c\ d\mu(c).
\end{equation}
\end{definition}

If $\mu\ge0$ satisfies the first identity in \eqref{repmeas}, then $0=\Delta u=6\scal xm$, so $m=0$ and $\mu$ satisfies \eqref{kolmrep}; thus representing measures are Kolmogorov measures in the sense of \Cref{def:harmrep}. The second identity in \eqref{repmeas} says that \emph{the multiplication operators of $\alg$ are averages of rank one projections onto the idempotent lines}; equivalently, the multiplication tensor is the third moment of a measure on the idempotent variety.

For a representing measure the trace correction in \eqref{waringtrace} vanishes, and the resulting positive Waring decomposition can always be made finite: the cubes $\scal x{\hat c}^3$, $c\in\Idm(\alg)$, form a compact subset of the space of cubic forms, of dimension $\binom{n+2}3$, so by Carath\'eodory's theorem (the elementary part of Tchakaloff's theorem \cite{Putinar-Tchakaloff,Bayer-Teichmann}) the existence of a representing measure is equivalent to the existence of a finite positive Waring decomposition
$$
u(x)=\sum_{k=1}^mw_k\scal x{\hat c_k}^3,\qquad w_k>0,\quad c_k\in\Idm(\alg),\quad m\le\tbinom{n+2}3 .
$$
We do not address the minimal length $m$ (the Waring rank problem).

\begin{theorem}\label{the:moments}
Let $\mu$ be a representing measure of $u$ supported on $\Idm(\alg)$. Then
\begin{align}
\int d\mu&=\tfrac{|c|}6\trace \Kill,&\int\hat c\,d\mu&=0,\label{mom01}\\
\int\hat c\otimes\hat c\,d\mu&=\tfrac{|c|}{6}\,\Kill,&\int\hat c\otimes\hat c\otimes\hat c\,d\mu&=\tfrac16\,\Tens.\label{mom23}
\end{align}
Moreover, every idempotent is in equilibrium with respect to $\mu$:
\begin{equation}\label{equil}
\int_{\Idm(\alg)}\scal{\hat c_0}{\hat c}^2\,\hat c\,d\mu(c)=M\,\hat c_0,\qquad c_0\in\Idm(\alg).
\end{equation}
\end{theorem}

\begin{proof}
The third moment is \eqref{repmeas}, the first moment was computed above, and the mass is \eqref{mass}. For \eqref{equil}, apply the second identity in \eqref{repmeas} to $x=y=\hat c_0$: the left hand side becomes $\hat c_0^{\,2}=c_0/|c|^2=\hat c_0/|c|$, while the right hand side is $6\int\scal{\hat c_0}{\hat c}^2\hat c\,d\mu$; and $\frac1{6|c|}=M$. For the second moment $\Sigma=\int\hat c\otimes\hat c\,d\mu$ we compute, using \eqref{repmeas} twice and then \eqref{equil},
\begin{align*}
\scal{Ax}{y}=\trace L(x)L(y)&=36\iint\scal x{\hat c}\scal y{\hat c'}\scal{\hat c}{\hat c'}^2\,d\mu(c)\,d\mu(c')\\
&=36\int\scal y{\hat c'}\Bigl\langle x,\int\scal{\hat c'}{\hat c}^2\hat c\,d\mu(c)\Bigr\rangle d\mu(c')=36M\scal{\Sigma x}{y}.
\end{align*}
Hence $\Sigma=\frac{1}{36M}\Kill=\frac{|c|}{6}\Kill$.
\end{proof}

\begin{corollary}\label{cor:design}
Let $\mu$ be a representing measure of $u$ supported on $\Idm(\alg)$. The following are equivalent:
\begin{enumerate}
\item[(i)] $\alg$ is an Einstein Hsiang algebra;
\item[(ii)] $\mu/\mu(\Idm(\alg))$ is a spherical $2$-design;
\item[(iii)] the symmetrized measure $\frac12(\mu+(-1)_*\mu)$ on the alternation set $\pm\hIdm$ is a spherical $3$-design.
\end{enumerate}
In particular, these properties do not depend on the choice of the representing measure.
\end{corollary}

\begin{proof}
A probability measure on $S$ is a $2$-design iff its first moment vanishes and its second moment is $\frac1n\Id$ \cite{Delsarte-Goethals-Seidel}. By \eqref{mom01}--\eqref{mom23} this happens iff $\Kill$ is a multiple of $\Id$. Symmetrization kills the odd moments and keeps the second one.
\end{proof}

\begin{remark}[geometric meaning]
By the quadratic trace identity of Hsiang algebras, the Einstein property (\Cref{def:einstein}) is equivalent to the vanishing of the mean curvature of $\Idm(\alg)$ in the sphere of radius $|c|$ \cite{Tkachev-book}.
Combining this with \Cref{cor:design} gives, for algebras which admit a representing measure on $\Idm(\alg)$,
\begin{align*}
\alg\ \text{is an Einstein Hsiang algebra}&\iff\Idm(\alg)\ \text{is minimal in the sphere}\\
&\iff\mu\ \text{is a $2$-design}\iff\pm\hIdm\ \text{carries a $3$-design}.
\end{align*}
For example, the alternation set of $x_1x_2x_3$ ($\Pol(1,2)$) is the cube, a $3$-design, and $\Kill=2\Id$; for the Lawson cubic ($\Pol(2,2)$) $\Kill=\mathrm{diag}(2,2,4,4)$, and its idempotent circle is not minimal.
\end{remark}

\begin{remark}[an energy]
Let
$$
E(\mu)=\iint\scal{\hat c}{\hat c'}^3\,d\mu(c)\,d\mu(c').
$$
By \eqref{repmeas}, $E(\mu)=\int u(\hat c)\,d\mu=M\mu(\Idm(\alg))$, and $\Fs uu=6E(\mu)$. Identity \eqref{equil} says that the potential $c_0\mapsto\int\scal{\hat c_0}{\hat c}^3d\mu(c)$ has a radial gradient at every point of $\Idm(\alg)$, i.e.\ every idempotent is a critical point of its own potential. This is a Fekete-type equilibrium condition, see \Cref{sec:open}.
\end{remark}

\section{Existence of representing measures}\label{sec:existence}

\subsection{Polar Hsiang cubics}
Let $u$ be the cubic \eqref{clifford} of a polar algebra $\alg\in\Pol(m_0,m_1)$, defined by a symmetric Clifford system $P_1,\dots,P_{m_0}$ on $\R{m_1}$, and put $\ell=\frac{m_1}2$. For a unit vector $\eta\in\R{m_0}$ let $E_+(\eta)=\ker(P(\eta)-\Id)$, a subspace of dimension $\ell$, and let $\Pi(\eta)=\frac12(\Id+P(\eta))$ be the orthogonal projection onto it. We denote by $\sigma_{m_0}$ the normalized surface measure on $S^{m_0-1}$, by $S_\eta$ the unit sphere of $E_+(\eta)$, and by $\sigma_\eta$ the normalized surface measure on $S_\eta$.

\begin{theorem}\label{the:clifford}
The normalized idempotents of $\alg$ form the sphere bundle
\begin{equation}\label{cliffidm}
\hIdm=\Bigl\{\Bigl(\tfrac1{\sqrt3}\,\eta,\ \sqrt{\tfrac23}\,\zeta\Bigr):\ \eta\in S^{m_0-1},\ \zeta\in E_+(\eta),\ |\zeta|=1\Bigr\}.
\end{equation}
Let $\nu$ be the probability measure on $\Idm(\alg)$ defined by
\begin{equation}\label{cliffnu}
\int f(\hat c)\,d\nu(c)=\int_{S^{m_0-1}}\Bigl(\int_{S_\eta}f\bigl(\tfrac1{\sqrt3}\,\eta,\sqrt{\tfrac23}\,\zeta\bigr)\,d\sigma_\eta(\zeta)\Bigr)d\sigma_{m_0}(\eta),\qquad f\in C(S^{n-1}),
\end{equation}
i.e.\ $\eta$ is chosen uniformly on $S^{m_0-1}$ and then $\zeta$ uniformly on $S_\eta$. Then
\begin{equation}\label{cliffwaring}
u(x)=\frac{\sqrt3\,m_0m_1}{4}\int\scal x{\hat c}^3\,d\nu(c),\qquad x\in\R{n},
\end{equation}
so that $\mu=\frac{\sqrt3}4m_0m_1\,\nu$ is a representing measure of $u$ and $u$ is stationary. Moreover $\Kill=m_1\Id_{\mathbb V_0}\oplus2m_0\Id_{\mathbb V_1}$, $\trace \Kill=3m_0m_1$, and $R(u)=\dfrac{2N_n}{27m_0m_1}$.
\end{theorem}

\begin{proof}
\emph{The idempotents.} Let $x=(y,z)\in S^{n-1}$ and $t=|y|$, so that $|z|^2=1-t^2$. Since $P(y)^2=|y|^2\Id$, the eigenvalues of $P(y)$ are $\pm|y|$, hence $\scal{P(y)z}z\le|y||z|^2$ and
$$
u(x)\le\tfrac12|y||z|^2=\tfrac12t(1-t^2)\le\tfrac1{3\sqrt3}=M,
$$
the last inequality being elementary (the maximum of $t(1-t^2)$ on $[0,1]$ is attained at $t=\frac1{\sqrt3}$). Equality holds throughout if and only if $t=\frac1{\sqrt3}$ and $P(y)z=|y|z$, i.e.\ $y=\frac1{\sqrt3}\eta$ with $|\eta|=1$ and $z=\sqrt{\frac23}\zeta$ with $\zeta\in S_\eta$. On the other hand, by \Cref{lem:critical} and \Cref{the:equi}(i), the maximum points of $u|_S$ are exactly the points $\hat c$, $c\in\Idm(\alg)$. This proves \eqref{cliffidm}.

\emph{The fibre integral.} Let $\hat c=(\frac1{\sqrt3}\eta,\sqrt{\frac23}\zeta)$ and $x=(y,z)\in\R{n}$. Then
$$
\scal x{\hat c}=a+b\scal z\zeta,\qquad a=\tfrac1{\sqrt3}\scal y\eta,\quad b=\sqrt{\tfrac23}.
$$
Fix $\eta$. The measure $\sigma_\eta$ is invariant under $\zeta\mapsto-\zeta$, so the odd powers of $\scal z\zeta$ have zero integral, while, since $\sigma_\eta$ is the normalized measure on the unit sphere of the $\ell$-dimensional space $E_+(\eta)$,
$$
\int_{S_\eta}\scal z\zeta^2\,d\sigma_\eta(\zeta)=\frac1\ell\,|\Pi(\eta)z|^2=\frac1{2\ell}\bigl(|z|^2+\scal{P(\eta)z}z\bigr).
$$
Hence, using $3ab^2=2a$,
\begin{equation}\label{cliff1}
\begin{aligned}
\int_{S_\eta}\scal x{\hat c}^3\,d\sigma_\eta(\zeta)&=a^3+3ab^2\int_{S_\eta}\scal z\zeta^2\,d\sigma_\eta(\zeta)\\
&=\frac{\scal y\eta^3}{3\sqrt3}+\frac{\scal y\eta}{\sqrt3\,\ell}\bigl(|z|^2+\scal{P(\eta)z}z\bigr).
\end{aligned}
\end{equation}

\emph{The base integral.} Now integrate \eqref{cliff1} against $\sigma_{m_0}$. The terms $\scal y\eta^3$ and $\scal y\eta|z|^2$ are odd in $\eta$ and have zero integral. For the remaining term, $\scal y\eta\scal{P(\eta)z}z=\sum_{i,j}y_i\eta_i\eta_j\scal{P_jz}z$, and $\int\eta_i\eta_j\,d\sigma_{m_0}=\frac{\delta_{ij}}{m_0}$ gives
$$
\int_{S^{m_0-1}}\scal y\eta\scal{P(\eta)z}z\,d\sigma_{m_0}(\eta)=\frac1{m_0}\scal{P(y)z}z=\frac2{m_0}\,u(x).
$$
Thus, by \eqref{cliffnu} and $\ell=\frac{m_1}2$,
$$
\int\scal x{\hat c}^3\,d\nu(c)=\frac{2}{\sqrt3\,\ell m_0}\,u(x)=\frac{4}{\sqrt3\,m_0m_1}\,u(x),
$$
which is \eqref{cliffwaring}.

\emph{The Killing operator.} In the same way, the mixed terms $\eta\otimes\zeta$ have zero integral, $\int_{S_\eta}\zeta\otimes\zeta\,d\sigma_\eta=\frac1\ell\Pi(\eta)$ and $\int_{S^{m_0-1}}\Pi(\eta)\,d\sigma_{m_0}=\frac12\Id$, so that
$$
\int\hat c\otimes\hat c\,d\nu=\frac1{3m_0}\Id_{\mathbb V_0}\oplus\frac2{3m_1}\Id_{\mathbb V_1}.
$$
Since $\mu=\frac{\sqrt3}4m_0m_1\,\nu$ is a representing measure and $|c|=\frac{\sqrt3}2$, \eqref{mom23} gives
$$
\tfrac{\sqrt3}{12}\,\Kill=\int\hat c\otimes\hat c\,d\mu=\tfrac{\sqrt3}{12}\bigl(m_1\Id_{\mathbb V_0}\oplus2m_0\Id_{\mathbb V_1}\bigr),
$$
which is the formula for $\Kill$, and $\trace \Kill=m_0m_1+2m_0m_1=3m_0m_1$. The value of $R(u)$ now follows from \eqref{ratio}.
\end{proof}

By the formula for $\Kill$, a polar algebra of type $(m_0,m_1)$ is an Einstein Hsiang algebra if and only if $m_1=2m_0$, i.e.\ $n=3m_0$, and then $\nu$ is a spherical $2$-design (\Cref{cor:design}). By the Hurwitz--Radon bound this forces $m_0\in\{1,2,4,8\}$: these are the four mutants \cite{Tkachev-book}. For the cubic $x_1x_2x_3$ ($\Pol(1,2)$) the measure $\nu$ is the uniform measure on the four vertices of a tetrahedron, and for the Lawson cubic ($\Pol(2,2)$) and the cubic of $\Pol(1,4)$ it is the arc length on one circle and on two circles, respectively (see \Cref{pro:lawson,ex:polar14}). The values of $\trace \Kill$ and $R(u)$ in \Cref{tab:ratio} for the polar algebras agree with \Cref{the:clifford}.

\subsection{Steiner decompositions and finite quadratures}\label{sec:psts}
The simplest Hsiang cubic $x_1x_2x_3$ ($\Pol(1,2)$) already has a finite idempotent quadrature (\Cref{ex:x1x2x3}): four equal atoms at the vertices of a tetrahedron. Many Hsiang cubics are assembled from copies of it along a combinatorial pattern, and then the quadrature is assembled in the same way.

\begin{definition}\label{def:psts}
A cubic form $u$ on $\R{n}$ admits a \emph{Steiner decomposition} if there are an orthonormal basis $e_1,\dots,e_n$, a family $B$ of $3$-element subsets $b=\{b_1,b_2,b_3\}\subset\{1,\dots,n\}$ (\emph{blocks}) such that any two distinct indices lie in at most one block, and signs $\varepsilon_b\in\{\pm1\}$, such that
\begin{equation}\label{psts}
u(x)=\sum_{b\in B}\varepsilon_b\,x_{b_1}x_{b_2}x_{b_3}.
\end{equation}
\end{definition}

In combinatorial terms, $(\{1,\dots,n\},B)$ is a partial Steiner triple system decorated by the signs $\varepsilon_b$; Hsiang algebras admitting such decompositions are studied in \cite{HsiangStrongRigidity}.

\begin{proposition}\label{pro:psts}
Let $u$ be a cubic form on $\R{n}$, not necessarily a Hsiang cubic, which admits a Steiner decomposition \eqref{psts}, and let $\alg=\alg(u)$ be its metrized algebra. Then:
\begin{enumerate}
\item[(i)] for every block $b$, the span of $e_{b_1},e_{b_2},e_{b_3}$ is a subalgebra of $\alg$ isomorphic to the algebra of $x_1x_2x_3$, and the four vectors of the set
$$
\mathcal I_b=\Bigl\{\tfrac12\bigl(s_1e_{b_1}+s_2e_{b_2}+s_3e_{b_3}\bigr):\ s\in\{\pm1\}^3,\ s_1s_2s_3=\varepsilon_b\Bigr\}
$$
are idempotents of $\alg$, called the \emph{block idempotents} of $b$; the sets $\mathcal I_b$ are pairwise disjoint, and $\mathcal I_B=\bigcup_{b\in B}\mathcal I_b$ has $4|B|$ elements;
\item[(ii)] with $\hat c=c/|c|$,
\begin{equation}\label{pstswaring}
u(x)=\frac{\sqrt3}{8}\sum_{c\in\mathcal I_B}\scal x{\hat c}^3 ,
\end{equation}
and the Killing operator is diagonal,
\begin{equation}\label{pstsA}
\Kill=2\operatorname{diag}\bigl(r(1),\dots,r(n)\bigr),\qquad \trace \Kill=6|B|,
\end{equation}
where the \emph{replication number} $r(i)$ is the number of blocks containing $i$.
\end{enumerate}
Consequently, if $u$ is a Hsiang cubic, then the measure with mass $\frac{\sqrt3}8$ at each point of $\mathcal I_B$ is a representing measure: $u$ admits a finite idempotent quadrature with equal weights, and $u$ is stationary.
\end{proposition}

\begin{proof}
The structure constants $\scal{e_ie_j}{e_k}=\partial_i\partial_j\partial_ku$ equal $\varepsilon_b$ if $\{i,j,k\}=b\in B$ and vanish otherwise; in particular $e_i^2=0$, since no monomial in \eqref{psts} contains a square. Because a pair $\{i,j\}$ lies in at most one block, $e_ie_j=\varepsilon_be_k$ for $\{i,j,k\}=b$, and $e_ie_j=0$ if $\{i,j\}$ lies in no block. Thus the span of $e_{b_1},e_{b_2},e_{b_3}$ is closed under multiplication, and the map $e_{b_1}\mapsto e_{b_1}$, $e_{b_2}\mapsto e_{b_2}$, $e_{b_3}\mapsto\varepsilon_be_{b_3}$ identifies it with the algebra of $x_1x_2x_3$. This proves the first claim of (i). For $c\in\mathcal I_b$ as in (i),
$$
c^2=\tfrac14\cdot2\sum_{\{i,j,k\}=b}s_is_j\,\varepsilon_b\,e_{k}=\tfrac12\sum_{k\in b}s_k\,e_k=c,
$$
since $s_is_j\varepsilon_b=s_is_js_1s_2s_3=s_k$. Two distinct blocks differ in at least one index, so their block idempotents have different supports, and the sets $\mathcal I_b$ are disjoint.

For (ii), fix a block $b$ and write $\hat c=\frac1{\sqrt3}(s_1e_{b_1}+s_2e_{b_2}+s_3e_{b_3})$ for $c\in\mathcal I_b$. Expanding $(s_1x_{b_1}+s_2x_{b_2}+s_3x_{b_3})^3$ and summing over the four admissible sign vectors $s$, the monomials $x_{b_i}^3$ and $x_{b_i}^2x_{b_j}$ cancel, because each of $s_i$ and $s_j$ takes the values $\pm1$ equally often on this set (with $s_i^3=s_i$, $s_i^2=1$), while the monomial $6s_1s_2s_3\,x_{b_1}x_{b_2}x_{b_3}=6\varepsilon_b\,x_{b_1}x_{b_2}x_{b_3}$ is the same for all four. Hence
$$
\sum_{c\in\mathcal I_b}\scal x{\hat c}^3=\frac{24}{3\sqrt3}\,\varepsilon_b\,x_{b_1}x_{b_2}x_{b_3},\qquad\text{i.e.}\qquad \varepsilon_b\,x_{b_1}x_{b_2}x_{b_3}=\frac{\sqrt3}8\sum_{c\in\mathcal I_b}\scal x{\hat c}^3,
$$
and summation over $b\in B$ gives \eqref{pstswaring}. Finally, for $x=\sum_ix_ie_i$ and fixed $j,k$ there is at most one $i$ with $\{i,j,k\}\in B$, so that
$$
\trace L(x)^2=\sum_{j,k}\Bigl(\sum_ix_i\scal{e_ie_j}{e_k}\Bigr)^2=\sum_ix_i^2\,\#\{(j,k):\{i,j,k\}\in B\}=2\sum_ir(i)\,x_i^2,
$$
which is \eqref{pstsA}; summing, $\trace \Kill=2\sum_ir(i)=6|B|$. If $u$ is a Hsiang cubic, all idempotents have length $|c|=\frac{\sqrt3}2$, which is the length of the block idempotents, and \eqref{pstswaring} is \eqref{repmeas}; its total mass $\frac{\sqrt3}2|B|=\frac{|c|}6\trace \Kill$ agrees with \eqref{mass}.
\end{proof}

\begin{example}\label{ex:psts}
(a) For $x_1x_2x_3$ ($\Pol(1,2)$) there is one block, and the four block idempotents form the tetrahedron of \Cref{ex:x1x2x3}.

(b) For $u=\det X$ on $M_3(\R{})$ the blocks are the permutations $\sigma\in S_3$, $b_\sigma=\{(1,\sigma1),(2,\sigma2),(3,\sigma3)\}$, with $\varepsilon_{b_\sigma}=\operatorname{sgn}\sigma$. The $24$ block idempotents are the matrices $\frac12P$, where $P$ runs over the signed permutation matrices with $\det P=1$, i.e.\ over the octahedral rotation group $O\subset SO(3)$. Thus the finite group $O$ replaces the invariant measure on $\Idm(\alg)=\frac12SO(3)$ of \Cref{the:homog}.

(c) For a mutant $\re\bigl((x_1x_2)x_3\bigr)$ on $\mathbb K_k^3$, in the standard basis of $\mathbb K_k$, the blocks are the triples $(e_i,e_j,e_l)$ with $\re\bigl((e_ie_j)e_l\bigr)\ne0$; there are $k^2$ of them and $4k^2$ block idempotents.

(d) For the tripling $u=\scal{x_1\times x_2}{x_3}$ of a vector product algebra $(\R{k},\times)$, $k\in\{3,7\}$ (\Cref{sec:peetre}), the blocks are the triples $(e_i,e_j,e_l)$ with $\scal{e_i\times e_j}{e_l}\ne0$: there are $k(k-1)$ of them. For $k=3$ this is (b) again, and for $k=7$ the $21$-dimensional exceptional algebra $\mathrm{trip}(\operatorname{Im}\mathbb O)$ has an idempotent quadrature with $168$ equal atoms. This algebra is neither polar nor covered by \Cref{the:homog}.
\end{example}

\begin{remark}[Steiner cubics which are not Hsiang]\label{rem:nonhsiang}
Parts (i) and (ii) of \Cref{pro:psts} hold for every cubic form $u$ with a Steiner decomposition: the block idempotents have length $\frac{\sqrt3}2$, the value $m=u(\hat c)=3^{-3/2}$ at them is the same, and \eqref{pstswaring} is a positive quadrature on $4|B|$ tetrahedral nodes. Moreover, $u$ is harmonic, and the proof of \Cref{the:ratio} gives $\|u\|_2^2=\trace \Kill /(6N_n)=|B|/N_n$. Hence, with $M=\max_Su\ge m$,
\begin{equation}\label{nonhsiang}
R(u)=\frac{N_n\,M^2}{|B|}=\frac{N_n}{27|B|}\Bigl(\frac Mm\Bigr)^2\ \ge\ \frac{N_n}{27|B|},
\end{equation}
with equality if and only if the block idempotents are maximum points of $u|_S$. For a Hsiang cubic this is automatic by \Cref{the:equi}. In general it fails: for the Fano plane, $u=\sum x_ix_jx_k$ over its seven lines on $\R7$, the point $x=(1,\dots,1)/\sqrt7$ gives $M\ge7^{-1/2}$, so $(M/m)^2\ge\frac{27}7$; the block idempotents are not maximum points, and the tetrahedral quadrature carries no information about the first variation of $R$.

The combinatorics alone is not sufficient: the sign decoration $(\varepsilon_b)$ of the blocks, which is determined by the algebra and not by the Steiner system, matters as well \cite{HsiangStrongRigidity}. Thus the Steiner system provides the scale $N_n/(27|B|)$, and the Hsiang identity, through the signs, is what makes this scale the actual value of $R$.
\end{remark}

A Steiner decomposition is called \emph{regular} if all replication numbers are equal, $r(i)=r$; then $|B|=nr/3$. All decompositions in \Cref{ex:psts} are regular, with $r=1,2,k,k-1$ respectively.

\begin{corollary}\label{cor:pstsreg}
Let $u$ be a Hsiang cubic with a regular Steiner decomposition with replication number $r$. Then $\alg$ is an Einstein Hsiang algebra with $\Kill=2r\Id$,
\begin{equation}\label{pstsR}
R(u)=\frac{(n+2)(n+4)}{9r},
\end{equation}
and the $8|B|=\frac83nr$ points $\pm\hat c$, $c\in\mathcal I_B$, form a spherical $3$-design on $S^{n-1}$.
\end{corollary}

\begin{proof}
If $r(i)=r$ for all $i$, then \eqref{pstsA} gives $\Kill=2r\Id$, so $\alg$ is Einstein, and $\trace \Kill=2rn$. Hence, by \eqref{ratio},
$$
R(u)=\frac{2N_n}{9\trace \Kill}=\frac{2n(n+2)(n+4)}{18rn}=\frac{(n+2)(n+4)}{9r}.
$$
By \Cref{pro:psts}, the measure $\mu$ with equal masses at the points of $\mathcal I_B$ is a representing measure of $u$. Since $\alg$ is Einstein, \Cref{cor:design} shows that the symmetrized measure $\frac12(\mu+(-1)_*\mu)$, which has equal masses at the $8|B|$ points $\pm\hat c$, $c\in\mathcal I_B$, is a spherical $3$-design. These points are pairwise distinct: the $\hat c$ with $c\in\mathcal I_B$ are distinct by \Cref{pro:psts}(i), and $\hat c\ne-\hat c'$ for $c,c'\in\mathcal I_B$: for different blocks the supports differ, and within one block the condition $s_1s_2s_3=\varepsilon_b$ is not preserved by $s\mapsto-s$.
\end{proof}

The Peirce profile determines $r$. If $n_{-1/2}=3d+2$, which is the case for the mutants and for all exceptional algebras with $n_{-1/2}\ne0$, the second trace identity of \cite{Tkachev-book} gives $\Kill=2(n_{-1}+d+1)\Id$, so that, by \eqref{pstsA} and the dimension formula $n=3n_{-1}+2n_{-1/2}-1$,
\begin{equation}\label{pstsr}
r=n_{-1}+d+1,\qquad n=3(r+d),\qquad |B|=r(r+d),
\end{equation}
cf.\ \cite{HsiangStrongRigidity}. Hence a Hsiang cubic with a regular Steiner decomposition has exactly $4r(r+d)$ block idempotents and $R(u)=\frac{(n+2)(n+4)}{3(n-3d)}$. \Cref{tab:psts} lists the Steiner decompositions of the Hsiang algebras known to exist.

\begin{table}[ht]\small
\renewcommand\arraystretch{1.25}
\begin{tabular}{c|c|l|c|c|c|c}
$n$ & $\mathsf{pr}(\alg)$ & Steiner system & $|B|$ & $r$ & atoms $4|B|$ & $R(u)$\\\hline
$3k$ & $(k-1,2,2k-2)$ & Latin square of order $k$ & $k^2$ & $k$ & $4k^2$ & $\frac{(3k+2)(3k+4)}{9k}$\\\hline
9 & $(0,5,3)$ & $\mathrm{GQ}(2,1)$ & 6 & 2 & 24 & $\frac{143}{18}$\\
15 & $(0,8,6)$ & $\mathrm{GQ}(2,2)$ & 15 & 3 & 60 & $\frac{323}{27}$\\
27 & $(0,14,12)$ & $\mathrm{GQ}(2,4)$ & 45 & 5 & 180 & $\frac{899}{45}$\\\hline
18 & $(1,8,8)$ & doubling of $\mathrm{GQ}(2,1)$ & 24 & 4 & 96 & $\frac{110}9$\\
30 & $(1,14,14)$ & doubling of $\mathrm{GQ}(2,2)$ & 60 & 6 & 240 & $\frac{544}{27}$\\
54 & $(1,26,26)$ & doubling of $\mathrm{GQ}(2,4)$ & 180 & 10 & 720 & $\frac{1624}{45}$\\\hline
21 & $(4,5,11)$ & tripling of the Fano plane & 42 & 6 & 168 & $\frac{575}{54}$\\\hline
12 & $(1,5,5)$ & none & -- & -- & -- & $\frac{224}{27}$
\end{tabular}
\medskip
\caption{Steiner decompositions and finite idempotent quadratures of the Hsiang algebras of \Cref{ex:psts} and of the exceptional Hsiang algebras with $n_{-1/2}\ne0$ known to exist \cite{HsiangStrongRigidity}. The first row is the family of mutants $\Pol(k,2k)$, $k\in\{1,2,4,8\}$ ($k=1$ is $x_1x_2x_3$); the next rows are the algebras with $n_{-1}=0$ ($n=9$: $\det X$), their doublings with $n_{-1}=1$ ($n=18$: $\re\det z$ on $M_3(\mathbb C)$), the tripling $\mathrm{trip}(\operatorname{Im}\mathbb O)$, and the algebra $\re\det z$ on $\mathrm{Sym}_3(\mathbb C)$. The values of $R(u)$ follow from \eqref{pstsR}, and for $n=12$ from \eqref{ratio} with $\Kill=6\Id$.}\label{tab:psts}
\end{table}

The Steiner systems in \Cref{tab:psts} are classical objects. For $n_{-1}=0$ ($n=9,15,27$) the blocks are the lines of the generalized quadrangles $\mathrm{GQ}(2,t)$, $t=1,2,4$: the $3\times3$ grid, the quadrangle $W(2)$ with $15$ points and $15$ lines, and the quadrangle of order $(2,4)$ with $27$ points and $45$ lines \cite{Payne-Thas}. For $n_{-1}=1$ ($n=18,30,54$) the system is obtained from the previous one by \emph{doubling}: the cubic form $\re u(x+\sqrt{-1}\,y)$ on $\R{n}\times\R{n}$ replaces every block $\{i,j,k\}$ by the four blocks of a Pasch configuration,
$$
x_ix_jx_k\ \longmapsto\ x_ix_jx_k-x_iy_jy_k-y_ix_jy_k-y_iy_jx_k ,
$$
so that $|B|$ is multiplied by $4$ and $r$ by $2$. For $\mathrm{trip}(\operatorname{Im}\mathbb O)$ the blocks are the triples $\{(i,1),(j,2),(k,3)\}$ with $\{i,j,k\}$ a line of the Fano plane, in all six orders; the same construction applied to the single line of $\{1,2,3\}$ gives the $3\times3$ grid of $\det X$. By a theorem announced in \cite{HsiangStrongRigidity}, every block of the Steiner decomposition of an exceptional Hsiang algebra lies in exactly $n_{-1}$ Pasch configurations; in \Cref{tab:psts} this number is $0$, $1$ and $4$ for the three exceptional families, and $k-1$ for the mutants. The $12$-dimensional algebra has no Steiner decomposition (announced in \cite{HsiangStrongRigidity}) and is covered by \Cref{cor:n12} instead.

The values of $R(u)$ in \Cref{tab:psts} agree with \Cref{tab:ratio}. Conversely, \eqref{pstsr} predicts the size of the finite quadrature, and the value of $R$, for every hypothetical exceptional Hsiang algebra with a regular Steiner decomposition, in terms of its Peirce profile alone.

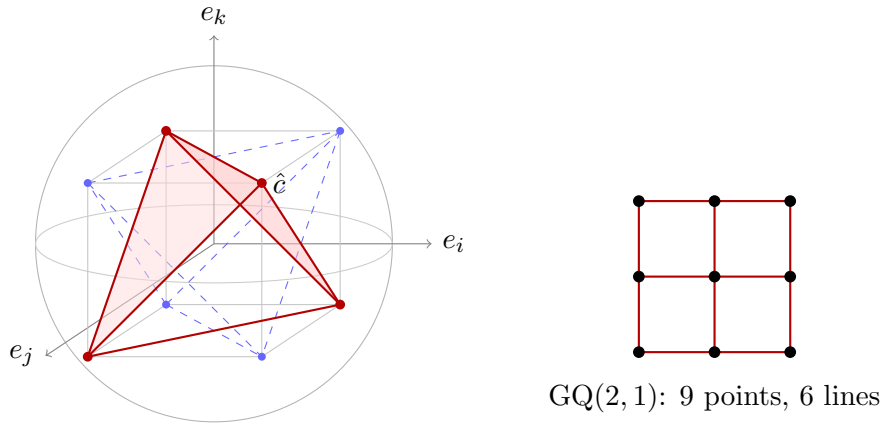
\begin{figure}[ht]
\centering
\begin{tikzpicture}[x={(1cm,0cm)},y={(-0.45cm,-0.3cm)},z={(0cm,1cm)},scale=1.15,line join=round]
\begin{scope}[x={(1cm,0cm)},y={(0cm,1cm)}]\draw[gray!60] (0,0) circle[radius=2.05]; \draw[gray!40] (0,0) ellipse[x radius=2.05,y radius=0.45];\end{scope}
\draw[->,gray] (0,0,0) -- (2.5,0,0) node[black,font=\small,right]{$e_i$};
\draw[->,gray] (0,0,0) -- (0,4.3,0) node[black,font=\small,left]{$e_j$};
\draw[->,gray] (0,0,0) -- (0,0,2.4) node[black,font=\small,above]{$e_k$};
\foreach \a in {-1,1}{\foreach \b in {-1,1}{
  \draw[gray!45] (-1,\a,\b)--(1,\a,\b); \draw[gray!45] (\a,-1,\b)--(\a,1,\b); \draw[gray!45] (\a,\b,-1)--(\a,\b,1);}}
\coordinate (m1) at (-1,-1,-1); \coordinate (m2) at (-1,1,1); \coordinate (m3) at (1,-1,1); \coordinate (m4) at (1,1,-1);
\foreach \p/\q in {m1/m2,m1/m3,m1/m4,m2/m3,m2/m4,m3/m4} \draw[dashed,blue!60] (\p)--(\q);
\coordinate (c1) at (1,1,1); \coordinate (c2) at (1,-1,-1); \coordinate (c3) at (-1,1,-1); \coordinate (c4) at (-1,-1,1);
\fill[red!25,opacity=0.5] (c1)--(c2)--(c4)--cycle;
\fill[red!15,opacity=0.5] (c1)--(c3)--(c4)--cycle;
\foreach \p/\q in {c1/c2,c1/c3,c1/c4,c2/c3,c2/c4,c3/c4} \draw[thick,red!70!black] (\p)--(\q);
\foreach \p in {c1,c2,c3,c4} \fill[red!70!black] (\p) circle[radius=1.6pt];
\foreach \p in {m1,m2,m3,m4} \fill[blue!60] (\p) circle[radius=1.3pt];
\node[font=\small,right] at (c1) {$\hat c$};
\end{tikzpicture}\qquad
\begin{tikzpicture}[scale=1.0]
\foreach \i in {0,1,2}{\draw[thick,red!70!black] (0,\i)--(2,\i); \draw[thick,red!70!black] (\i,0)--(\i,2);}
\foreach \i in {0,1,2}{\foreach \j in {0,1,2}{\fill (\i,\j) circle[radius=2.2pt];}}
\node[font=\small] at (1,-0.6) {$\mathrm{GQ}(2,1)$: $9$ points, $6$ lines};
\end{tikzpicture}
\caption{Left: an atom, i.e.\ the four nodes $\hat c$, $c\in\mathcal I_b$ (solid), and their antipodes (dashed). Right: the Steiner system $\mathrm{GQ}(2,1)$ of $\det X$; each line is an atom.}\label{fig:atom}
\end{figure}

\smallskip\noindent\textit{Geometric interpretation.} Every block $b=\{i,j,k\}$ is an \emph{atom} of the quadrature \eqref{pstswaring}: the four nodes $\hat c$, $c\in\mathcal I_b$, are the vertices of a regular tetrahedron inscribed in the unit sphere of the coordinate $3$-space $\operatorname{span}(e_i,e_j,e_k)$, and their antipodes $-\hat c$ form the dual tetrahedron, so that together they are the eight vertices of a cube (\Cref{fig:atom}, left). The quadrature of \Cref{pro:psts} thus puts equal weights at the vertices of $|B|$ such tetrahedra, and the Steiner system describes how they are glued: two atoms share a coordinate axis exactly when the corresponding blocks meet. For instance, the Steiner system of $\det X$ is the $3\times3$ grid $\mathrm{GQ}(2,1)$, whose six lines give six atoms and the $24$ nodes of \Cref{ex:psts}(b) (\Cref{fig:atom}, right). In this sense the incidence geometry of \Cref{tab:psts} is a discrete skeleton of $u$, in the same way as the idempotent variety is a continuous one (\Cref{sec:intro}).

\smallskip\noindent\textit{Relation to \Cref{pro:peetre}.} In \Cref{ex:psts}(c) and (d) every block meets each of the three groups of variables $x_1,x_2,x_3$ exactly once, so these are block-trilinear cubics in the sense of \Cref{def:blocktri} whose trilinear form has entries in $\{0,\pm1\}$. In this sense \Cref{pro:psts} and \Cref{pro:peetre} are two faces of the same combinatorial structure: the first gives the quadrature (first variation), the second the size of $\|T\|_2$ (the value of $R$).

\subsection{Homogeneous idempotent varieties}
Let $G$ be the group of orthogonal automorphisms of $\alg$, i.e.\ of orthogonal maps $g$ with $u\circ g=u$; it is a compact Lie group acting on $\Idm(\alg)$. Since $Z_{\hat c}(g^{-1}x)=Z_{g\hat c}(x)$, averaging over $G$ preserves \eqref{kolmrep} and \eqref{repmeas}.

\begin{theorem}\label{the:homog}
Assume that
\begin{enumerate}
\item[(a)] $G$ acts transitively on $\Idm(\alg)$; let $\nu$ be the $G$-invariant probability measure on $\Idm(\alg)$;
\item[(b)] the $G$-invariant cubic forms on $\R{n}$ are $\R{}u$.
\end{enumerate}
Then $\mu=\frac{|c|}6\trace \Kill\cdot\nu$ is a representing measure of $u$. In particular $u$ is stationary, and \Cref{the:moments} applies.
\end{theorem}

\begin{proof}
If $v\ne0$ were a $G$-fixed vector, then $|x|^2\scal xv$ would be a $G$-invariant cubic form which is not a multiple of the harmonic form $u$; by (b), $G$ has no non-zero fixed vector. The vector $m=\int\hat c\,d\nu$ is $G$-fixed, hence $m=0$. The cubic form $p(x)=\int\scal x{\hat c}^3\,d\nu(c)$ is $G$-invariant, so $p=\gamma u$ by (b), and by \eqref{waringtrace} with $m=0$ also $p=\int Z_{\hat c}\,d\nu$. Pairing with $u$ and using \Cref{lem:zonal},
$$
\gamma\|u\|_2^2=\int_Spu\,d\sigma=\frac6{N_n}\int u(\hat c)\,d\nu(c)=\frac{6M}{N_n}>0 .
$$
Hence $\mu=\gamma^{-1}\nu$ is a representing measure, and its mass is given by \eqref{mass}.
\end{proof}

\begin{corollary}\label{cor:cartan}
For Cartan's cubics $\Cartan d$, $d=1,2,4,8$ ($n=5,8,14,26$), the invariant measure on $\Idm(\alg)$ is (up to a factor) a representing measure, $u$ is stationary, and the measure is a spherical $2$-design.
\end{corollary}

\begin{proof}
Let $\mathscr H_3(\mathbb K_d)$ be the Jordan algebra of Hermitian $3\times3$ matrices over $\mathbb K_d\in\{\R{},\mathbb C,\mathbb H,\mathbb O\}$. Its automorphism group $K$ preserves the trace and $\trace X^3$, hence acts on the traceless part by automorphisms of $\alg$, and acts transitively on primitive idempotents, hence on $\Idm(\alg)$ \cite[Chap.~6]{Nadirashvili-Tkachev-Vladuts}, \cite{Tkachev-book}. By the spectral theorem, the $K$-invariant polynomials on $\mathscr H_3(\mathbb K_d)$ are generated by $\trace X,\trace X^2,\trace X^3$ \cite{Faraut-Koranyi}; on the traceless part the invariant cubics are therefore $\R{}\trace X^3=\R{}u$. This gives the hypotheses (a) and (b) of \Cref{the:homog}. The algebras of Cartan's cubics are Einstein Hsiang algebras \cite{Tkachev-book}, so \Cref{cor:design} applies.
\end{proof}

\begin{remark}
The hypotheses of \Cref{the:homog} also hold for $u=\det X$ on $M_3(\R{})$: the group $SO(3)\times SO(3)$, acting by $X\mapsto aXb^{T}$, consists of orthogonal automorphisms, acts transitively on $\Idm(\alg)=\frac12SO(3)$, and its only invariant cubic forms are the multiples of $\det X$. Indeed, by the singular value decomposition every $X$ is of the form $a\,\mathrm{diag}(s_1,s_2,s_3)\,b^T$ with $a,b\in SO(3)$, so an invariant cubic form $p$ is determined by the cubic polynomial $p(\mathrm{diag}(s_1,s_2,s_3))$, which is invariant under permutations of the $s_i$ and under simultaneous sign changes of two of them; the only such cubic polynomials are the multiples of $s_1s_2s_3$.
 This is an Einstein Hsiang algebra with $\Kill=4\Id$, so the invariant measure on $\frac12SO(3)$ is (up to a factor) a representing measure and a $2$-design. For the polar algebras, for instance the Lawson cubic ($\Pol(2,2)$), the identity component of $G$ has too many invariant cubic forms, and (b) fails unless the discrete part of $G$ is used.
\end{remark}

\begin{corollary}\label{cor:n12}
Let $u(z)=\re\det z$ on $\mathrm{Sym}_3(\mathbb C)$ with $\scal zw=\re\trace z\bar w$; this is the $12$-dimensional exceptional Hsiang algebra with $\mathsf{pr}(\alg)=(1,5,5)$ \cite[Sec.~6.11.6(iv)]{Nadirashvili-Tkachev-Vladuts}. Then the hypotheses of \Cref{the:homog} hold, and $u$ is stationary.
\end{corollary}

\begin{proof}
The group $SU(3)$, acting by $z\mapsto gzg^T$, and the complex conjugation $z\mapsto\bar z$ act by orthogonal automorphisms. By Takagi's factorization every $z$ is of the form $gDg^T$ with $g\in U(3)$ and $D$ diagonal with non-negative entries. Writing $g=e^{i\theta}g'$ with $g'\in SU(3)$, and using the diagonal matrices in $SU(3)$ and the signed permutation matrices of determinant $1$, we see that every $SU(3)$-orbit meets the diagonal matrices $\mathrm{diag}(a_1,a_2,a_3)$, and that an invariant cubic form $p$ restricts to a polynomial in $a_k,\bar a_k$ which is symmetric and invariant under $a_k\mapsto e^{i\psi_k}a_k$ with $\psi_1+\psi_2+\psi_3=0$. The only cubic monomials with this invariance are $a_1a_2a_3$ and $\bar a_1\bar a_2\bar a_3$, so $p=\alpha\re\det z+\beta\operatorname{Im}\det z$. Invariance under conjugation gives $\beta=0$, which is (b). For (a), if $|z|=1$ with Takagi values $s_k$, then $\re\det z\le s_1s_2s_3\le3^{-3/2}$, with equality if and only if $s_1=s_2=s_3$ and $\det z>0$, i.e.\ $z=3^{-1/2}gg^T$ with $g\in SU(3)$ (replace $g$ by $-g$ if necessary). Thus $\hIdm$ is a single $SU(3)$-orbit.
\end{proof}
This algebra is the only exceptional Hsiang algebra known to exist with $n_{-1/2}\ne0$ which has no Steiner decomposition (announced in \cite{HsiangStrongRigidity}), so \Cref{the:homog} and \Cref{pro:psts} complement each other.

\subsection{Explicit representing measures}
\begin{example}[$\re z^3$]
The normalized idempotents are $\hat c_k=(\cos\varphi_k,\sin\varphi_k)$, $\varphi_k=\frac{2\pi k}{3}$, $k=0,1,2$, and $\sum_k\cos^3(\varphi-\varphi_k)=\frac34\cos3\varphi$. Hence the uniform measure with total mass $4M$ is a representing measure. Since on $S^1$ all three zonal harmonics $Z_{\hat c_k}$ coincide with $\frac14\cos3\varphi$, \emph{every} positive measure of that total mass satisfies \eqref{kolmrep}: Kolmogorov measures are far from unique, while the representing measure on $\Idm(\alg)$ is unique.
\end{example}

\begin{example}[$x_1x_2x_3$, $\Pol(1,2)$]\label{ex:x1x2x3}
With $\hat c_\varepsilon=\varepsilon/\sqrt3$, $\varepsilon\in\{\pm1\}^3$, $\varepsilon_1\varepsilon_2\varepsilon_3=1$, expansion gives $\sum_\varepsilon(\varepsilon\cdot x)^3=24x_1x_2x_3$. Hence
$$
x_1x_2x_3=\frac{\sqrt3}{8}\sum_\varepsilon\scal x{\hat c_\varepsilon}^3,\qquad \mu(\Idm)=\frac{\sqrt3}2,\qquad \frac{\sqrt3}8\sum_\varepsilon\hat c_\varepsilon\otimes\hat c_\varepsilon=\frac{\sqrt3}{6}\Id=\frac{|c|}6\cdot2\Id,
$$
in accordance with \eqref{mom01}--\eqref{mom23} and $\Kill=2\Id$.
\end{example}

\begin{proposition}[Lawson cubic, $\Pol(2,2)$]\label{pro:lawson}
Let $\nu=d\varphi/4\pi$ be the normalized arc length on the circle \eqref{lawsonidm}. Then
\begin{equation}\label{lawsonwaring}
u(x)=\sqrt3\int\scal x{\hat c(\varphi)}^3\,d\nu(\varphi),
\end{equation}
so $\mu=\sqrt3\,\nu$ is a representing measure, and its second moment is $\frac{\sqrt3}6\mathrm{diag}(1,1,2,2)=\frac{|c|}6\Kill$.
\end{proposition}

\begin{proof}
In the coordinates $y=(v_0,v_1)$, $z=(v_2,v_3)$ the Lawson cubic is \eqref{clifford} with
$$
P_1=\begin{pmatrix}0&1\\1&0\end{pmatrix},\qquad P_2=\begin{pmatrix}1&0\\0&-1\end{pmatrix},
$$
i.e.\ $(m_0,m_1)=(2,2)$. In \eqref{lawsonidm}, $\eta=(\sin\varphi,\cos\varphi)$, and $\varphi$, $\varphi+2\pi$ give the two unit vectors of $E_+(\eta)$; hence the measure $\nu$ of \Cref{the:clifford} is $d\varphi/4\pi$, and the claims follow from \Cref{the:clifford} with $\frac{\sqrt3}4m_0m_1=\sqrt3$.
\end{proof}

\begin{example}[the cubic of $\Pol(1,4)$]\label{ex:polar14}
For $u=\frac12v_0(|z|^2-|w|^2)$ on $\R{}\times\mathbb C\times\mathbb C$, see \eqref{pol14}, the idempotent variety consists of the two circles $\hat c=(\frac1{\sqrt3},\sqrt{\frac23}e^{i\varphi},0)$ and $\hat c=(-\frac1{\sqrt3},0,\sqrt{\frac23}e^{i\varphi})$. By \Cref{the:clifford}, $\frac{\sqrt3}{2}(\nu_++\nu_-)$, where $\nu_\pm$ is the normalized arc length on the two circles, is a representing measure, with total mass $\sqrt3=\frac{\sqrt3}{12}\trace \Kill$, and $u$ is stationary.
\end{example}

\section{Local minimality}\label{sec:local}

We now analyse the directions in which the first variation \eqref{danskin} vanishes.

Let $D=\{h\in\hcub:\ h|_{\hIdm}=0\}$. The infinitesimal rotations of $u$,
$$
o(u)=\{h_X(x)=\scal{Du(x)}{Xx}:\ X\in\so(n)\},
$$
always lie in $D\cap u^\perp$: indeed, since $X$ is skew-symmetric,
$$
h_X(\hat c)=\tfrac12\scal{\hat c^{\,2}}{X\hat c}=\tfrac1{2|c|}\scal{\hat c}{X\hat c}=0,
$$
and $h_X\perp u$ because rotations preserve $\|u\|_2$. Moreover $\dim o(u)=\dim\so(n)-\dim\mathrm{Der}_{\so}(\alg)$, where $\mathrm{Der}_{\so}(\alg)$ is the Lie algebra of $G$.

\begin{theorem}\label{the:strict}
Let $u$ be a Hsiang cubic which admits a Kolmogorov measure $\mu$ with $\supp\mu=\Idm(\alg)$, and assume $D=o(u)$. Then $u$ is a strict local minimizer of $R$ modulo rotations and scaling: there is a neighbourhood $U$ of $u$ in $\hcub$ such that $R(p)\ge R(u)$ for $p\in U$, with equality only if $p=t\,u\circ g$ for some $t>0$ and $g\in O(n)$.
\end{theorem}

\begin{proof}
Let $W=u^\perp\cap o(u)^\perp\subset\hcub$. For $h\in W\setminus\{0\}$ we have $h\notin D$; by \eqref{cubature}, $\int h(\hat c)\,d\mu=0$, and since $h$ does not vanish identically on $\supp\mu=\Idm(\alg)$, it takes a positive value there. Hence $\phi(h)=\max_ch(\hat c)>0$ on $W\setminus\{0\}$, and by homogeneity and compactness $\phi(h)\ge\delta\|h\|_2$ for some $\delta>0$. Since $(u+h)(\hat c)=M+h(\hat c)$ exactly,
$$
R(u+h)\ge\frac{(M+\delta\|h\|_2)^2}{\|u\|_2^2+\|h\|_2^2}>\frac{M^2}{\|u\|_2^2}=R(u)\qquad\text{for }0<\|h\|_2<\varepsilon_0.
$$
The map $\Phi(t,g,h)=t\,(u+h)\circ g$ from $\R{}_{>0}\times O(n)\times W$ to $\hcub$ has surjective differential at $(1,\Id,0)$, since its image is $\R{}u+o(u)+W=\hcub$. By the inverse function theorem, every $p$ near $u$ can therefore be written as $p=t\,(u+h)\circ g$ with $h\in W$ small, and $R(p)=R(u+h)$.
\end{proof}

\begin{corollary}\label{cor:strict}
The cubic $x_1x_2x_3$ on $\R3$ ($\Pol(1,2)$) and the Lawson cubic on $\R4$ ($\Pol(2,2)$) are strict local minimizers of $R$ on $\hcub$ modulo rotations and scaling. Consequently
$$
\kappa_3\le\tfrac{35}9,\qquad \kappa_4\le\tfrac{32}9,
$$
and each of these values is a strict local minimum value of $R$.
\end{corollary}

\begin{proof}
We apply \Cref{the:strict}. Representing measures with full support on $\Idm(\alg)$ are the four equal atoms on the tetrahedron for $x_1x_2x_3$ (\Cref{ex:x1x2x3}) and the arc length for the Lawson cubic (\Cref{pro:lawson}). It remains to verify $D=o(u)$, i.e.\ $\dim D=\dim o(u)$.

\emph{$x_1x_2x_3$.} Here $G$ is finite, so $\dim o(u)=3$. The evaluation map $\hcub\to\R{4}$, $h\mapsto(h(\hat c_\varepsilon))_\varepsilon$, is onto: on the four points $\hat c_\varepsilon=\varepsilon/\sqrt3$ the harmonic cubic $x_1x_2x_3$ is constant, and the zonal harmonics $Z_{e_i}$, $i=1,2,3$, are non-zero multiples of $\varepsilon_i$; the four resulting vectors are the rows of a Hadamard matrix. Hence $\dim D=7-4=3$.

\emph{Lawson.} A direct computation gives $\dim o(u)=5$ (so the Lie algebra of $G$ is one-dimensional), and $\dim\hcub=16$. Using the rational parametrization $\cos\frac\varphi2=\frac{1-t^2}{1+t^2}$, $\sin\frac\varphi2=\frac{2t}{1+t^2}$ of \eqref{lawsonidm}, the condition $h|_{\hIdm}=0$ becomes the vanishing of two polynomials in $t$ with rational coefficients (the coefficients of $1$ and $\frac1{\sqrt2}$). The resulting linear map $\hcub\to\mathbb Q[t]^2$ has rank $11$, as elimination over $\mathbb Q$ shows; hence $\dim D=5$.
\end{proof}

\begin{remark}[Cartan's cubic]\label{rem:cartanlocal}
For $n=5$ local minimality is superseded by \Cref{the:cartanmin}: Cartan's cubic $\Cartan1$ is the global minimizer of $R$ on $\R5$, unique up to rotations and scaling, so that $\kappa_5=3$. The criterion of \Cref{the:strict} is consistent with this. The invariant measure of \Cref{cor:cartan} has full support, and a representation-theoretic computation with respect to $SO(3)\subset G$ shows that $D=o(u)$: both are the $7$-dimensional irreducible $SO(3)$-module.
\end{remark}

\subsection{The cubic of \texorpdfstring{$\Pol(1,4)$}{Pol(1,4)} is not a local minimizer}
Let now $u=\frac12v_0(y_1^2+y_2^2-y_3^2-y_4^2)$ be the cubic \eqref{pol14} of $\Pol(1,4)$ on $\R5=\R{}\times\mathbb C\times\mathbb C$, and write $z=y_1+iy_2$, $w=y_3+iy_4$, so that $u=\frac12v_0(|z|^2-|w|^2)$. Its idempotents form two circles,
$$
\Idm(\alg)=\bigl\{(\tfrac12,z,0):\ |z|^2=\tfrac12\bigr\}\cup\bigl\{(-\tfrac12,0,w):\ |w|^2=\tfrac12\bigr\},
$$
and $u$ is stationary (\Cref{ex:polar14}). The natural competitor of $u$ lives on the hyperplane $e_0^\perp=\mathbb C^2$:
\begin{equation}\label{Tcubic}
u_L=\tfrac12\re(z\,w^2)=\tfrac12y_1(y_3^2-y_4^2)-y_2y_3y_4 .
\end{equation}
In the coordinates $(v_0,v_1,v_2,v_3)=(-y_2,y_1,y_3,y_4)$ of $e_0^\perp$ this is exactly the normalized Lawson cubic ($\Pol(2,2)$) of \Cref{sec:hsiang}. Note that $u_L$, regarded as a cubic form on $\R5$, is not a Hsiang cubic (the direction $e_0$ annihilates the whole algebra), but it is harmonic, $\max_{S^4}u_L=\max_{S^3}u_L=M$, and $u_L$ vanishes on $\Idm(\alg)$, being linear in $z$ and quadratic in $w$.

The two cubics $u$ and $u_L$ are related through Cartan's cubic $\Cartan1$. Writing $\Cartan1=\sqrt2\det X$ on $\Sym_0(3,\R{})$ in coordinates adapted to the splitting of a $3\times3$ symmetric matrix into blocks of sizes $2+1$, one finds, after an orthogonal change of the variables $y$,
\begin{equation}\label{cartansplit}
\sqrt3\,\Cartan1=\sqrt6\det X=\tfrac32\,u+\sqrt3\,u_L-\tfrac13\bigl(v_0^3-\tfrac34v_0|y|^2\bigr),
\end{equation}
where the last term is a multiple of the zonal harmonic $Z_{e_0}$. Thus, up to a zonal correction, \emph{Cartan's cubic is the sum of the cubic of $\Pol(1,4)$ and the Lawson cubic}, placed in complementary positions: $u$ is even in $z$, while $u_L$ is linear in $z$ and quadratic in $w$.

The key observation is that the cubic of $\Pol(1,4)$ and the Lawson cubic equioscillate jointly.

\begin{lemma}\label{lem:joint}
On $S^4$,
\begin{equation}\label{joint}
|u+i\,u_L|\le M,
\end{equation}
and equality at a point $x$ implies $u(x)=0$ or $u_L(x)=0$.
\end{lemma}

\begin{proof}
The torus $\mathbb T^2$ acts orthogonally on $\R{}\times\mathbb C\times\mathbb C$ by $(v_0,z,w)\mapsto(v_0,e^{i\alpha}z,e^{i\beta}w)$. It preserves $u$ and rotates the phase of $zw^2$. Therefore the joint amplitude $|u+iu_L|$ is controlled by the orbit invariants $a=|z|^2$, $b=|w|^2$, $c=v_0^2$, which range over the simplex $a+b+c=1$:
$$
4|u+iu_L|^2=c\,(a-b)^2+\re(zw^2)^2\le c\,(a-b)^2+ab^2 .
$$
Compare this with $u$ on its meridian $w=0$: at the point with $|z|^2=m$, $v_0^2=1-m$ it takes the value $p(m)=\frac12m\sqrt{1-m}$, and $\max_{[0,1]}p=M$, attained only at $m=\frac23$ (this is the equioscillation of $u$, \Cref{the:equi}). We claim that the joint amplitude is dominated by this meridian profile:
$$
c\,(a-b)^2+ab^2\le 4p(m)^2=m^2(1-m),\qquad m=\max(a,b).
$$
Indeed, if $a\ge b$ then $(a-b)^2\le a^2$ and $ab^2\le a^2b$, so the left hand side is at most $a^2(c+b)=a^2(1-a)$; if $b\ge a$ then $(a-b)^2\le b^2$ and the left hand side is at most $b^2(c+a)=b^2(1-b)$. Hence $|u+iu_L|\le p(m)\le M$. Equality forces $m=\frac23$ and equality in the elementary estimates, i.e.\ $(a,b,c)\in\{(\frac23,0,\frac13),(0,\frac23,\frac13),(\frac13,\frac23,0)\}$. In the first two cases $u_L(x)=0$, in the third $u(x)=0$.
\end{proof}

\begin{proposition}\label{pro:polarsaddle}
For every $\theta\notin\frac\pi2\mathbb Z$,
$$
R(\cos\theta\,u+\sin\theta\,u_L)<R(u)=R(u_L).
$$
In particular, the cubic $u$ of $\Pol(1,4)$ is stationary, but it is not a local minimizer of $R$.
\end{proposition}

\begin{proof}
First, $u\perp u_L$ in $L^2(\sigma)$, because $u$ is even and $u_L$ is odd under $z\mapsto-z$, and $\|u_L\|_2=\|u\|_2$ by \Cref{lem:fischer}, since $\Fs{u_L}{u_L}=\Fs uu=2$. Hence $p_\theta=\cos\theta\,u+\sin\theta\,u_L$ satisfies $\|p_\theta\|_2=\|u\|_2$ for all $\theta$, and $p_\theta=\re\bigl(e^{-i\theta}(u+iu_L)\bigr)$. By \Cref{lem:joint}, $p_\theta\le M$ on $S^4$, and equality at $x$ would require $|u+iu_L|(x)=M$ and $\arg(u+iu_L)(x)=\theta$; for $\theta\notin\frac\pi2\mathbb Z$ this contradicts the second part of the lemma. By compactness, and since $p_\theta$ is odd, $\|p_\theta\|_\infty<M$, so $R(p_\theta)<M^2/\|u\|_2^2=R(u)$. For $\theta=\frac\pi2$ we get $R(u_L)=R(u)$, since $\max_{S^4}u_L=M$.
\end{proof}

\begin{remark}
Thus on the great circle $\{p_\theta\}$ of the sphere $\{\|p\|_2=\|u\|_2\}\subset\hcub$ through the cubic of $\Pol(1,4)$ and the Lawson cubic, the Chebyshev ratio attains its maximum $R(u)$ exactly at $\pm u$ and $\pm u_L$ and is strictly smaller everywhere else. In particular, not every Hsiang cubic is a local minimizer of the Chebyshev problem. By \eqref{cartansplit}, the direction of decrease from $u$ points towards Cartan's cubic $\Cartan1$, the global minimizer in $\R5$ (\Cref{the:cartanmin}), and it does so through the Lawson cubic.\end{remark}

\section{Concluding remarks}\label{sec:open}

The central structural question raised by \Cref{the:kolm,the:moments} is the following.

\begin{problem}\label{prob:waring}
Does every exceptional Hsiang algebra admit a representing measure supported on $\Idm(\alg)$, i.e.\ a positive measure $\mu$ on $\Idm(\alg)$ with
$$
xy=6\int_{\Idm(\alg)}\scal x{\hat c}\scal y{\hat c}\,\hat c\,d\mu(c),\qquad x,y\in\alg\,?
$$ Is the normalized Riemannian volume of $\Idm(\alg)$ (suitably weighted) always such a measure?
\end{problem}

By Theorem~D the answer is positive for all polar algebras, for the algebras covered by \Cref{the:homog} and for all Hsiang algebras with a Steiner decomposition (\Cref{pro:psts}). Together these include every Hsiang algebra known to exist, so the problem concerns the hypothetical exceptional algebras in the admissible Peirce profiles that are neither realized nor excluded; any such algebra with a Steiner decomposition in the sense of \cite{HsiangStrongRigidity} is automatically stationary. The second question is open even for the known algebras. A positive answer would give the identity ``the algebra is the third moment of its idempotent variety'' for all Hsiang algebras and, by \Cref{the:kolm}, would prove that every Hsiang cubic is stationary for the Chebyshev problem. The equilibrium condition \eqref{equil} is a necessary condition that can be tested intrinsically; it is natural to ask whether it characterizes the representing measures.

\begin{problem}\label{prob:global}
Prove \Cref{conj:kappa}. Is $\kappa_n$ attained by a Hsiang cubic for every $n$? If so, by the one maximizing $\trace \Kill$ among the Hsiang cubics in $\R n$?
\end{problem}

Some numerical experiments by the author support \Cref{conj:kappa} for $n=3$: a random search over $\hcub(\R3)$, followed by local minimization from many random starting points, did not produce values of $R$ below $\frac{35}9$, and the best values found approached $\frac{35}9$ at cubics congruent to $x_1x_2x_3$ ($\Pol(1,2)$). The experiments also suggest that the landscape of $R$ is not simple: besides the absolute minimum there appear to be several other local minimum levels, which makes a global proof by local methods alone unlikely.

\begin{problem}\label{prob:converse}
(a) If a harmonic cubic has only the critical values $\pm M$ and $0$ on the sphere, is it a Hsiang cubic? (b) If a harmonic cubic is stationary for \eqref{kappa0} and its extremal set is infinite, is it a Hsiang cubic?
\end{problem}

\subsection{Further directions}
The quadrature identity \eqref{cubature} represents the functional $h\mapsto\int_Sh\,u\,d\sigma$ on $\hcub$ by a positive measure on a lower-dimensional set, as for quadrature domains \cite{Gustafsson-Putinar-book,Shapiro-schwarz,Gustafsson-Shapiro,Khavinson-Lundberg}. This suggests the \emph{full} moment problem: which harmonic polynomials of higher degree admit quadrature identities with nodes on $\hIdm$, and do the subalgebras $\R{}c\oplus\alg_c(-1)$, $\R{}c\oplus\alg_c(-\frac12)$ produce families of them?

There is also an energy interpretation. By the reproducing property of $Z_a$,
$$
\|u\|_2^2=\frac6{N_n}\iint G\bigl(\scal{\hat c}{\hat c'}\bigr)\,d\mu(c)\,d\mu(c'),\qquad G(t)=t^3-\frac{3}{n+2}t,
$$
for every Kolmogorov measure $\mu$, so at fixed sup-norm the Chebyshev problem is a \emph{maximal energy problem} for measures on the alternation set, a counterpart of ``Chebyshev constant $=$ transfinite diameter'' \cite{Saff-Totik}. For finite $\Idm(\alg)$ the idempotents are indeed Fekete-type configurations (an equilateral triangle, a regular tetrahedron for $x_1x_2x_3$), and in general $\scal{c_1}{c_2}\ge-\frac38$ with equality exactly on triads \cite{Tkachev-book}.

Finally, \Cref{rem:kellogg} identifies \eqref{kappa0} with a minimal Lipschitz problem, and one may ask whether Hsiang cubics are distinguished for other $L^\infty$-variational problems on harmonic polynomials, or whether the equation $\Delta_\infty u=\lambda|x|^2u$ has natural solutions in higher degree; for degree four the Cartan--M\"unzner polynomials \cite{Munzner-II} are the obvious candidates.

\bibliographystyle{plain}

\end{document}